\documentclass[reqno,12pt]{amsart}
\usepackage{amsmath,amsfonts,amsthm,amssymb,indentfirst}
\usepackage[all,poly]{xy} 
\usepackage{color}
\usepackage[pagebackref,colorlinks]{hyperref}
\usepackage{kantlipsum}
\usepackage[all,poly]{xy}
\usepackage{mathrsfs}
\usepackage{enumerate}

\theoremstyle{plain}
\newtheorem{lemma}{Lemma}[section] 
\newtheorem{theorem}[lemma]{Theorem}
\newtheorem{corollary}[lemma]{Corollary}
\newtheorem{proposition}[lemma]{Proposition}

\theoremstyle{definition}
\newtheorem{remark}[lemma]{Remark}
\newtheorem{example}[lemma]{Example}
\newtheorem{definition}[lemma]{Definition}

\newcommand{\Zset}{\mathbb Z}

\newcommand{\M}{\operatorname{\mathbb M}}
\newcommand{\so}{\mathbf{s}}
\newcommand{\ra}{\mathbf{r}}
\newcommand{\V}{\mathcal V}

\newcommand{\Ter}{\operatorname{Ter}}
\newcommand{\Sink}{\operatorname{Sink}}
\newcommand{\NE}{\operatorname{NE}}
\newcommand{\EC}{\operatorname{EC}}

\newcommand{\ol}{\overline}

\title[Porcupine-quotient graphs and composition series]{Porcupine-quotient graphs, the fourth primary color, and graded composition series of Leavitt path algebras}

\author{Lia Va\v s}
\address{Department of Mathematics, Saint Joseph's University, Philadelphia, PA 19131, USA}
\email{lvas@sju.edu}



\begin{document}

\begin{abstract}
If $E$ is a directed graph, $K$ is a field, and $I$ is a graded ideal of the Leavitt path algebra $L_K(E),$ then $I$ is completely determined by a pair $(H,S)$ of two sets of vertices of $E,$ called an admissible pair, and one writes $I=I(H,S)$ in this case. The ideal $I$ is graded isomorphic to the Leavitt path algebra of the {\em porcupine graph} of $(H,S)$ and the quotient $L_K(E)/I$ is graded isomorphic to the Leavitt path algebra of the {\em quotient graph} of $(H,S).$ We present a construction which generalizes both the porcupine and the quotient constructions and enables one to consider quotients of graded ideals: if $(H,S)$ and $(G,T)$ are admissible pairs such that $(H,S)\leq (G,T)$ (in the sense which corresponds exactly to $I(H,S)\subseteq I(G,T)$), we define the {\em porcupine-quotient graph} $(G,T)/(H,S)$ such that its Leavitt path algebra is graded isomorphic to the quotient $I(G,T)/I(H,S).$ 

Using the porcupine-quotient construction, the existence of a graded composition series of $L_K(E)$ is equivalent to the existence of a finite increasing chain of admissible pairs of $E,$ starting with the trivial pair and ending with the improper pair, such that the quotient of two consecutive pairs is cofinal (a graph is cofinal exactly when its Leavitt path algebra is graded simple). We characterize the existence of such a chain with a set of conditions on $E$ which also provides an algorithm for obtaining a composition series. The conditions are presented in terms of four types of vertices which are all ``terminal'' in a certain sense. Three of the four types are often referred to as the three primary colors of Leavitt path algebras. The fourth primary color in the title of this paper refers to the fourth type of vertices. As a corollary of our results, every unital Leavitt path algebra has a graded composition series. 

We show that the existence of a composition series of $E$ is equivalent to the existence of a suitably defined composition series of the graph monoid $M_E$ as well as a composition series of the talented monoid $M_E^\Gamma.$ We also show that an ideal of $M_E^\Gamma$ is minimal exactly when it is generated by the element of $M_E^\Gamma$ corresponding to a terminal vertex. We characterize graphs $E$ such that only one or only two out of three possible types (periodic, aperiodic, or incomparable) appear among the composition factors of $M_E^\Gamma.$ 
\end{abstract}

\maketitle

\section{Introduction}
If $E$ is a directed graph and $K$ a field
the Leavitt path algebra $L_K(E)$ is naturally graded by the group of integers. The lattice of graded $L_K(E)$-ideals corresponds to the lattice of pairs of certain sets of vertices called the admissible pairs 
(we review the relevant definition in section \ref{subsection_quotient_and_porcupine}). The ideal $I(H,S)$ corresponding to an admissible pair $(H,S)$ is graded isomorphic to the Leavitt path algebra of a graph introduced in \cite{Lia_porcupine} which is called the porcupine graph.  The porcupine graph resembles the older construction of a hedgehog graph (\cite[Definitions 2.5.16 and 2.5.20]{LPA_book}) except that the ``spines'' added to the ``body'' determined by $H\cup S$ are longer 
(Example \ref{example_porcupine_and_quotient}  illustrates this), so the name ``porcupine'' was chosen to reflect that. While the Leavitt path algebra of the hedgehog of $(H,S)$ is isomorphic to $I(H,S),$ this isomorphism does not have to be graded. In contrast, the Leavitt path algebra of the porcupine of $(H,S)$ is {\em graded} isomorphic to $I(H,S).$

One can also define the quotient graph $E/(H,S)$ (\cite[Definition 2.4.14]{LPA_book})
in such a way that the quotient $L_K(E)/I(H,S)$ is graded isomorphic to the Leavitt path algebra of $E/(H,S).$ In section \ref{section_porcupine_quotient}, we introduce a graph construction which generalizes both the porcupine and the quotient graph constructions and enables one to represent the quotient of two graded ideals as the Leavitt path algebra of this newly defined graph. Specifically, if $(H,S)$ and $(G,T)$ are admissible pairs such that $(H,S)\leq (G,T)$ (in the sense which corresponds exactly to $I(H,S)\subseteq I(G,T)$), we define the porcupine-quotient graph $(G,T)/(H,S)$ (Definition \ref{definition_porcupine_quotient}) and show that its Leavitt path algebra is graded isomorphic to the quotient $I(G,T)/I(H,S)$ (Theorem \ref{theorem_porcupine_quotient}).  

We also consider two pre-ordered monoids, $M_E$ and $M_E^\Gamma,$ originated in relation to some classification questions (see, for example, \cite{Ara_Goodearl},\cite{Ara_Pardo_graded_K_classification}, \cite{Eilers_et_al}, and \cite{Roozbeh_Annalen}). The {\em graph monoid} $M_E$ is isomorphic to the monoid $\V(L_K(E))$ of the isomorphism classes of finitely generated projective modules. The natural grading of a Leavitt path algebra induces an action of the infinite cyclic group $\Gamma=\langle t\rangle \cong \Zset$ on the graded isomorphism classes of finitely generated graded projective $L_K(E)$-modules and there is a $\Gamma$-isomorphism of the monoid $\V^\Gamma(L_K(E))$ of such graded isomorphism classes and the monoid $M_E^\Gamma,$ also known as the {\em talented monoid} or the {\em graph $\Gamma$-monoid}. In particular, the following lattices are isomorphic: the lattice of order-ideals of $M_E,$ the lattice of $\Gamma$-order-ideals of $M_E^\Gamma,$ the lattice of graded ideals of $L_K(E),$ and the lattice of admissible pairs of $E$. By Proposition \ref{proposition_monoids_of_porcupine_quotient}, if $(G,T)/(H,S)$ is the porcupine-quotient graph of two admissible pairs of  $E,$ then $M_{(G,T)/(H,S)}$ is isomorphic to the quotient of the order-ideals corresponding to $(G,T)$ and $(H,S)$ and  $M_{(G,T)/(H,S)}^\Gamma$ is isomorphic to the quotient of the $\Gamma$-order-ideals corresponding to $(G,T)$ and $(H,S).$ 

We say that $L_K(E)$ has a graded composition series if there is a finite and increasing chain of graded ideals, starting with the trivial ideal and ending with the improper ideal, such that the quotient of each two consecutive ideals is graded simple. Since a Leavitt path algebra is graded simple if and only if the underlying graph is cofinal (see section \ref{subsection_graphs} for a review of this concept), Theorem \ref{theorem_porcupine_quotient} enables us to relate the existence of a graded composition series of $L_K(E)$ with the existence of a finite and increasing chain of admissible pairs, starting with the trivial pair and ending with the improper pair, such that the porcupine-quotient of two consecutive pairs is cofinal. If such a chain exists, we say that $E$ has a composition series. Theorem \ref{theorem_porcupine_quotient} and Proposition \ref{proposition_monoids_of_porcupine_quotient} imply Corollary \ref{corollary_composition_equivalent_with_talented} stating that the following conditions are equivalent. 
\begin{center}
\begin{tabular}{ll}
(1) $E$ has a composition series.$\hskip2cm$ & (2) $L_K(E)$ has a graded composition series. \\
(3) $M_E$ has a composition series. & (4) $M_E^\Gamma$ has a composition series.
\end{tabular}
\end{center}
We aim to characterize the existence of the above composition series by a set of conditions on $E$ which can be directly checked and which produce a specific composition series and achieve that in Theorem \ref{theorem_comp_series}. In order to obtain this result, we start with section \ref{section_graded_simple} in which we introduce a type of vertices which are ``terminal'' in the same sense as the vertices of any of the three types below. 
\begin{enumerate}
\item A {\em sink} is a vertex which emits no edges. A sink connects to no other vertex in the graph except, trivially, to itself.  
\item A {\em cycle without exits} is a 
cycle whose vertices emit only one edge to another vertex in the cycle. The vertices in such a cycle do not connect to any vertices outside of the cycle.  
\item An {\em extreme cycle} is a cycle such that the range of every exit from the cycle connects back to a vertex in the cycle. The vertices in such a cycle $c$ connect only to the vertices on cycles in the same ``cluster'' as $c.$
\end{enumerate}
The significance of these three groups of vertices lies in the fact that the Leavitt path algebra of a finite graph is {\em graded simple} exactly when there is a unique ``cluster'' of vertices of one of the three types above. Because of this, the three graphs below are the three quintessential examples of graphs with the above three types of vertices. The authors of \cite{LPA_book} refer to the Leavitt path algebras of these three graphs as the {\em three primary colors} of Leavitt path algebras.  
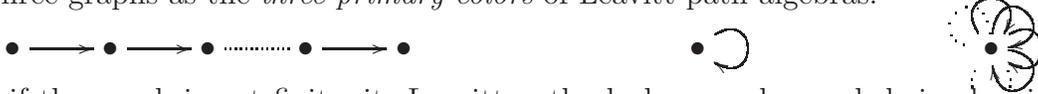
\[\xymatrix{\bullet\ar[r] &
\bullet\ar[r] & \bullet\ar@{.}[r] & \bullet\ar[r] & \bullet}\hskip3.45cm\xymatrix{\bullet\ar@(ur,dr)}
\hskip3.45cm\xymatrix{\bullet \ar@(ur,dr) \ar@(u,r) \ar@(ul,ur) \ar@{.} @(l,u) \ar@{.} @(dr,dl) \ar@(r,d)& }\]
However, if the graph is not finite, its 
Leavitt path algebra can be graded simple without having exactly one cluster of the three types of vertices as above. For example, the Leavitt path algebras of the graph below is graded simple and the graph has neither cycles nor sinks. 
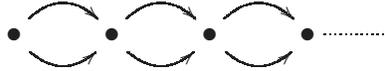
\[\xymatrix{\bullet\ar@/_1pc/ [r] \ar@/^1pc/ [r] &\bullet\ar @/_1pc/ [r] \ar@/^1pc/ [r]&
\bullet\ar @/_1pc/ [r] \ar@/^1pc/ [r]& \bullet\ar@{.}[r]&}\]

In Definition \ref{definition_terminal_path}, we 
introduce terminal paths as the infinite paths whose vertices are terminal in the same sense as the above three types. According to this definition, every infinite path of the above graph is terminal. In Definition \ref{definition_clusters}, we make the concept of a ``cluster'' more formal. In Theorem \ref{theorem_graded_simple}, we characterize graded simplicity of a Leavitt path algebra $L_K(E)$ by a set of conditions on $E$ which are direct to check and which are given in terms of the existence of exactly one cluster of the four types of terminal vertices. The existence of the fourth type does not contradict the Trichotomy Principle (\cite[Proposition 3.1.14]{LPA_book}), but it refines it: it distinguishes between sinks and terminal paths. 
The results of the last two sections illustrate that this distinction is a useful one. 

In Theorem \ref{theorem_comp_series}, we present a set of conditions on $E$ which are equivalent with $E$ having a composition series. Such conditions are constructive in the following sense: given a graph, one can construct a chain of admissible pairs such that the porcupine-quotient graphs of two consecutive pairs are cofinal and check if such a chain terminates after finitely many steps. Informally, such a chain is obtained by iteratively cutting the terminal vertices (and their breaking sets if $E$ is not row-finite). A direct corollary of Theorem \ref{theorem_comp_series} is that every unital Leavitt path algebra has a graded composition series (Corollary \ref{corollary_unital_comp_series}). 

Using the natural order $\leq$ and the action of $\Gamma$ on $M_E^\Gamma,$ one can categorize each element of $M_E^\Gamma$ as exactly one of the following three types: periodic, aperiodic and incomparable. If all nonzero elements of a $\Gamma$-order-ideal $I$ of $M_E^\Gamma$ have the same type, $I$ is said to be of that type also. In Theorem \ref{theorem_periodic_aperiodic_incomp}, we show that a $\Gamma$-order-ideal $I$ of $M_E^\Gamma$ is minimal exactly when $I$ is generated by the element $[v]$ of $M_E^\Gamma$ corresponding to a terminal vertex $v$ and that $I$ is periodic (respectively, aperiodic or comparable) exactly when $[v]$ is such also. In Theorem \ref{theorem_two_types} and  Corollary \ref{corollary_one_type}, we characterize  graphs $E$ such that only two or only one  of those three types appear among the composition factors of $M_E^\Gamma.$ In one of these cases, our work generalizes results from \cite{Roozbeh_Alfilgen_Jocelyn} formulated only for finite graphs.

\section{Prerequisites}\label{section_prerequisites}

\subsection{Graded rings}
A ring $R$ (not necessarily unital) is {\em graded} by a group $\Gamma$ if $R=\bigoplus_{\gamma\in\Gamma} R_\gamma$ for additive subgroups $R_\gamma$ and if $R_\gamma R_\delta\subseteq R_{\gamma\delta}$ for all $\gamma,\delta\in\Gamma.$ The elements of the set $\bigcup_{\gamma\in\Gamma} R_\gamma$ are said to be {\em homogeneous}. A left ideal $I$ of a graded ring $R$ is {\em graded} if $I=\bigoplus_{\gamma\in \Gamma} I\cap R_\gamma.$ Graded right ideals and graded ideals are defined similarly. A graded ring is {\em graded simple} if there are no nontrivial and proper two-sided graded ideals (note that we do not require it to be graded Artinian). 

A ring $R$ is an involutive ring, or a $*$-ring, if there is an anti-automorphism $*:R\to R$ of order two. If $R$ is also a $K$-algebra for some commutative $*$-ring $K$, then $R$ is a $*$-algebra if $(kx)^*=k^*x^*$ for all $k\in K$ and $x\in R.$
If $R$ is a $\Gamma$-graded ring with involution, it is a {\em graded $*$-ring} if $R_\gamma^*\subseteq R_{\gamma^{-1}}.$

A ring $R$ is {\em locally unital} if for every finite set $F\subseteq R,$ there is an idempotent $u\in R$ such that $xu=ux=x$ for every $x\in F.$ A $\Gamma$-graded ring $R$ is {\em graded locally unital} if for every finite set $F\subseteq R$ (of homogeneous elements) there is a homogeneous idempotent $u\in R$ such that $xu=ux=x$ for every $x\in F.$ The statements with and without the part in parenthesis are equivalent. 

\subsection{Graphs and properties of vertex sets}\label{subsection_graphs}
If $E$ is a directed graph, we let $E^0$ denote the set of vertices, $E^1$ denote the set of edges, and $\so$ and $\ra$ denote the source and the range maps of $E.$ A {\em sink} of $E$ is a vertex which emits no edges and an {\em infinite emitter} is a vertex which emits infinitely many edges. A vertex of $E$ is {\em regular} if it is neither a sink nor an infinite emitter. The graph $E$ is {\em row-finite} if it has no infinite emitters and $E$ is {\em finite} if it has finitely many vertices and edges.  

A {\em path} is a single vertex or a sequence of edges $e_1e_2\ldots e_n$ for some positive integer $n$ such that $\ra(e_i)=\so(e_{i+1})$ for $i=1,\ldots, n-1.$  The length $|p|$ of a path $p$ is zero if $p$ is a vertex and it is $n$ if $p$ is a sequence of $n$ edges. The set of vertices on a path $p$ is denoted by $p^0.$ 

The functions $\so$ and $\ra$ extend to paths naturally. A path $p$ is {\em closed}  if $\so(p)=\ra(p).$ A {\em cycle} is a closed path such that different edges in the path have different sources. A cycle has {\em an exit} if a vertex on the cycle emits an edge which is not an edge of the cycle. A cycle $c$ is {\em extreme} if $c$ has exits and for every path $p$ with $\so(p)\in c^0,$ there is a path $q$ such that $\ra(p)=\so(q)$ and $\ra(q)\in c^0.$  

An {\em infinite path} is a sequence of edges $e_1e_2\ldots$ such that $\ra(e_n)=\so(e_{n+1})$ for $n=1,2\ldots.$ Just as for finite paths, we use $p^0$ for the set of vertices of an infinite path $p.$ To emphasize that a path is infinite, we denote it by a Greek letter in sections \ref{section_graded_simple} to \ref{section_talented}.   

Let $E^{\leq\infty}$ be the set of infinite paths or finite paths ending in a sink or an infinite emitter. A vertex $v$ is {\em cofinal} if for each $p\in E^{\leq\infty}$ there is $w\in p^0$ such that $v\geq w$ and  $E$ is {\em cofinal} if each vertex is cofinal.

If $u,v\in E^0$ are such that there is a path $p$ with $\so(p)=u$ and $\ra(p)=v$, we write $u\geq v.$ For $V\subseteq E^0,$
the set $T(V)=\{u\in E^0\mid v\geq u$ for some $v\in V\}$ is called the {\em tree} of $V,$ and, following \cite{Lia_irreducible}, we use $R(V)$ to denote the set $\{u\in E^0\mid u\geq v$ for some $v\in V\}$ called the {\em root} of $V.$ To emphasize that the tree and the root of $V$ are considered in the graph $E$, we use $T^E(V)$ and $R^E(V).$ If $V=\{v\},$ we use $T(v)$ for $T(\{v\})$ and $R(v)$ for $R(\{v\}).$ The requirement that a cycle $c$ with an exit is extreme can be written as $T(c^0)\subseteq R(c^0)$ (compare with the requirement in Definition \ref{definition_terminal_path}).

A subset $H$ of $E^0$ is said to be {\em hereditary} if $T(H)\subseteq H.$ The set $H$ is {\em saturated} if $v\in H$ for any regular vertex $v$ such that $\ra(\so^{-1}(v))\subseteq H.$ For every $V\subseteq E^0,$ the intersection of all saturated sets of vertices which contain $V$ is the smallest saturated set which contains $V.$ This set is the {\em saturated closure of $V$}. The saturated closure $\ol{V}$ of $T(V)$ is both hereditary and saturated and it is the smallest hereditary and saturated set which contains $V.$ 

The saturated closure of $V$ is the union of the sets $\Lambda_n(V), n=0,1,\ldots,$ defined by $\Lambda_0(V)=V$ and $\Lambda_{n+1}(V)=\Lambda_n(V)\cup \{v\in E^0\mid v$ is regular and $\ra(\so^{-1}(v))\subseteq \Lambda_n(V) \}.$ The proof is analogous to the proof of \cite[Lemma 2.0.7]{LPA_book}: if $\Lambda(V)$ denotes the union $\bigcup_{n=0}^\infty\Lambda_n(V),$ it is direct to check that $\Lambda$ is saturated, that it contains $V,$ and that it is contained in every saturated set which contains $V.$  This description is used in the proof of the next lemma.

\begin{lemma}
Let $E$ be any graph, $V\subseteq E^0,$ and $H\subseteq E^0$ be a hereditary set such that $\ol V\subseteq H\subseteq R(T(V)).$ For $v\in H,$ let $P_v(T(V))$ be the set of paths originating at $v$ and terminating at a vertex of $T(V)$ such that no vertex, except the range, is in $T(V).$ The following conditions are equivalent. 
\begin{enumerate}
\item $H=\ol{V}.$

\item The set $H-T(V)$ does not contain infinite emitters and every infinite path with vertices in $H$ contains a vertex of $T(V).$

\item The set $P_v(T(V))$ is finite for every $v\in H.$
\end{enumerate}
\label{lemma_saturated_closure} 
\end{lemma}
\begin{proof}
The implication (1) $\Rightarrow$ (2)
follows directly from the description of $\ol{V}$ in terms of $\Lambda_n(T(V)).$ 

The contrapositive of the implication (2) $\Rightarrow$ (3) is rather direct since if $P_v(T(V))$ is infinite for some $v\in H,$ then there is either an infinite emitter on some of the paths in $P_v(T(V))$ or there is an infinite path with all of its vertices in $H-\ol{V}$. 

To show (3) $\Rightarrow$ (1), assume that (3) holds and let $n_v=\max\{|p|\mid p\in P_v(T(V))\}$ for $v\in H.$ If $n_v=0,$ then $v\in T(V)\subseteq \ol{V}.$ If $n_v>0,$ then $v$ is regular and, for each $e\in \so^{-1}(v),$ $\ra(e)\in H$ and  $n_{\ra(e)}<n_v.$ By induction, we can conclude that $\ra(e)\in \ol{V}.$ As $\ol{V}$ is saturated, $v\in \ol{V}.$  This shows that $H\subseteq \ol{V}.$ As the other direction is assumed to hold, (1) holds.  
\end{proof}

\subsection{Leavitt path algebra}
If $K$ is any field, the \emph{Leavitt path algebra} $L_K(E)$ of $E$ over $K$ is a free $K$-algebra generated by the set  $E^0\cup E^1\cup\{e^\ast\mid e\in E^1\}$ such that for all vertices $v,w$ and edges $e,f,$

\begin{tabular}{ll}
(V)  $vw =0$ if $v\neq w$ and $vv=v,$ & (E1)  $\so(e)e=e\ra(e)=e,$\\
(E2) $\ra(e)e^\ast=e^\ast\so(e)=e^\ast,$ & (CK1) $e^\ast f=0$ if $e\neq f$ and $e^\ast e=\ra(e),$\\
(CK2) $v=\sum_{e\in \so^{-1}(v)} ee^\ast$ for each regular vertex $v.$ &\\
\end{tabular}

The elements of $L_K(E)$ are of the form $\sum_{i=1}^n k_ip_iq_i^\ast$ for some $n$, paths $p_i$ and $q_i$, and $k_i\in K,$ for $i=1,\ldots,n$ where $v^*=v$ for $v\in E^0$ and $p^*=e_n^*\ldots e_1^*$ for a path $p=e_1\ldots e_n.$ The algebra 
$L_K(E)$ is an involutive $K$-algebra with 
$\left(\sum_{i=1}^n k_ip_iq_i^\ast\right)^*=\sum_{i=1}^n k_i^*q_ip_i^\ast$ where $k_i\mapsto k_i^*$ is any involution on $K$. In addition, $L_K(E)$ is graded locally unital (with the finite sums of vertices as the local units), and $L_K(E)$ is unital if and only if $E^0$ is finite in which case $\sum_{v\in E^0}v$ is the identity.

If we consider $K$ to be trivially graded by $\Zset,$ $L_K(E)$ is naturally graded by $\Zset$ so that the $n$-component $L_K(E)_n$ is the $K$-linear span of the elements $pq^\ast$ for paths $p, q$ with $|p|-|q|=n.$ This grading and the involutive structure make $L_K(E)$ into a graded $*$-algebra.

If $R$ is a $K$-algebra which contains  elements $p_v$ for $v\in E^0,$ and $x_e$ and $y_e$ for $e\in E^1$ such that the five axioms hold for these elements, the Universal Property of $L_K(E)$ states that there is a unique algebra homomorphism $\phi:L_K(E)\to R$ such that $\phi(v)=p_v, \phi(e)=x_e,$ and $\phi(e^*)=y_e$ (see \cite[Remark 1.2.5]{LPA_book}). If $R$ is $\Zset$-graded and $p_v\in R_0$ for $v\in E^0,$ $x_e\in R_1$ and $y_e\in R_{-1}$ for $e\in E^1,$ then $\phi$ is graded.
By the Graded Uniqueness Theorem (\cite[Theorem 2.2.15]{LPA_book}), such graded map $\phi$ is injective if $p_v\neq 0$ for $v\in E^0.$ If $R$ is involutive and $\phi$ is such that $y_e=x_e^*,$ then $\phi$ is a $*$-homomorphism (i.e., $\phi(x^*)=\phi(x)^*$ for every $x\in L_K(E)$).

\subsection{The quotient and the porcupine graphs}\label{subsection_quotient_and_porcupine}
If $H$ is hereditary and saturated, a {\em breaking vertex} of $H$ is an element of the set 
\[B_H=\{v\in E^0-H\,|\, v\mbox{ is an infinite emitter and }0<|\so^{-1}(v)\cap \ra^{-1}(E^0-H)|<\infty\}.\]
For each $v\in B_H,$ let $v^H$ stands for $v-\sum ee^*$ where the sum is taken over $e\in \so^{-1}(v)\cap \ra^{-1}(E^0-H).$

An {\em admissible pair} is a pair $(H, S)$ where $H\subseteq E^0$ is hereditary and saturated and $S\subseteq B_H.$ 
For an admissible pair $(H,S)$, 
the ideal $I(H,S)$ generated by $H\cup \{v^H \,|\, v\in S \}$ is graded since it is generated by homogeneous elements. It is the $K$-linear span of the elements $pq^*$ for paths $p,q$ with $\ra(p)=\ra(q)\in H$ and the elements $pv^Hq^*$ for paths $p,q$ with $\ra(p)=\ra(q)=v\in S$ (see \cite[Lemma 5.6]{Tomforde}). Conversely, for a graded ideal $I$, $H=I\cap E^0$ is hereditary and saturated and for $S=\{v\in B_H\mid v^H\in I\},$ $I=I(H,S)$ (\cite[Theorem 5.7]{Tomforde}, also \cite[Theorem 2.5.8]{LPA_book}). If $S=\emptyset,$ we shorten $(H,\emptyset)$ to $H$ and $I(H, \emptyset)$ to $I(H).$ 

The set of admissible pairs is a lattice with respect to the relation 
\[(H,S)\leq (G,T)\;\; \mbox{ if }H\subseteq K\mbox{ and }S\subseteq G\cup T\]
(see \cite[Proposition 2.5.6]{LPA_book} for the meet and the join of this lattice). The correspondence $(H,S)\mapsto I(H,S)$ is a lattice isomorphism of this lattice and the lattice of graded ideals. 

An admissible pair $(H,S)$ gives rise to the {\em quotient graph} $E/(H,S),$ defined so that 
\[
\begin{array}{l}
(E/(H,S))^0=E^0-H \cup\{v'\mid v\in B_H-S\}, \\
(E/(H,S))^1=\{e\in E^1\mid \ra(e)\notin H\}\cup\{e'\mid e\in E^1\mbox{ and }\ra(e)\in B_H-S\}, 
\end{array}
\]
and with $\so$ and $\ra$ the same as in $E$ on $E^1 \cap (E/(H,S))^1$ and $\so(e')=\so(e),$ $\ra(e')=\ra(e)'.$ The algebras $L_K(E)/I(H,S)$ and $L_K(E/(H,S))$ are graded isomorphic (see \cite[Theorem 5.7]{Tomforde}).

An admissible pair $(H,S)$ also gives rise to the {\em porcupine graph} $P_{(H,S)}$  
defined as follows. Let
\[
\begin{array}{l}
F_1(H,S)=\{e_1\ldots e_n\mbox{ is a path of }E\mid \ra(e_n)\in H, \so(e_n)\notin H\cup S\}\mbox{ and}\\
F_2(H,S)=\{p\mbox{ is a path of }E\mid \ra(p)\in S,\; |p|>0\}.
\end{array}\]
For each $e\in (F_1(H,S)\cup F_2(H,S))\cap E^1,$ let $w^e$ be a new vertex and $f^e$ a new edge such that $\so(f^e)=w^e$ and $\ra(f^e)=\ra(e).$
Continue this process inductively as follows. 
For each path $p=eq$ where $q\in F_1(H,S)\cup F_2(H,S)$ and $|q|>0,$ add a new vertex $w^p$ and a new edge $f^p$ such that $\so(f^p)=w^p$ and $\ra(f^p)=w^q.$ One defines the vertices and edges of $P_{(H,S)}$ as follows 
\[
\begin{array}{l}
P_{(H,S)}^0=H\cup S\cup \{w^p \mid p\in F_1(H,S)\cup F_2(H,S)\}\mbox{ and}\\
P_{(H,S)}^1= \{e\in E^1\,|\, \so(e)\in H\}\cup \{e\in E^1\,|\, \so(e)\in S, \ra(e)\in H\}\cup \{f^p\mid p\in F_1(H,S)\cup F_2(H,S)\}.
\end{array}
\]
The $\so$ and $\ra$ maps are the same as in $E$ for the common edges and they are defined as above for the new edges.  
The algebras $L_K(P_{(H,S)})$ and $I(H,S)$  are graded isomorphic (see \cite[Theorem 3.3]{Lia_porcupine}).

We exhibit some examples of porcupine and quotient graphs below. Example \ref{example_three_porcupine_quotients}  contains further examples of porcupine graphs. 

\begin{example}
\label{example_porcupine_and_quotient} 
Let $E$ be the first graph below, let $H=\{v\},$ and let $S=B_H=\{w\}$. 
In this case, the quotient graph is the second graph below. We have that $F_1(H,S)=\{e_3, e_2e_3, e_1e_2e_3\}$ and $F_2(H,S)=\{e_1\}.$ The porcupine graph is the third graph below. 
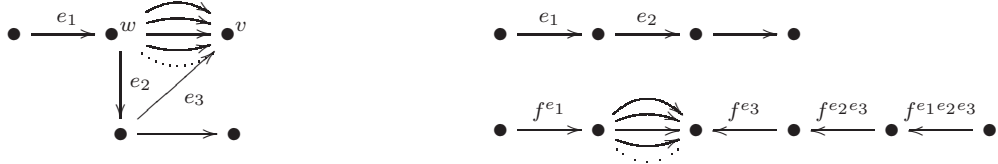
\[\xymatrix{
\bullet \ar[r]^{e_1} &{\bullet}^{w} \ar@{.} @/_1pc/ [r] _{\mbox{ } } \ar@/_/ [r] \ar [r] \ar@/^/ [r] \ar@/^1pc/ [r] \ar[d]^{e_2} 
& {\bullet}^{v}\\ & \bullet\ar[ur]_{e_3}\ar[r]&\bullet }\hskip3cm
\xymatrix{\bullet\ar[r]^{e_1}&\bullet\ar[r]^{e_2}&\bullet\ar[r] &\bullet &&\\\bullet \ar[r]^{f^{e_1}} &{\bullet} \ar@{.} @/_1pc/ [r] _{\mbox{ } } \ar@/_/ [r] \ar [r] \ar@/^/ [r] \ar@/^1pc/ [r]  & {\bullet} & \bullet\ar[l]_{f^{e_3}}& \bullet\ar[l]_{f^{e_2e_3}}&\bullet\ar[l]_{f^{e_1e_2e_3}} }\]

Next, let $E$ be the first graph below and let $H$ consists of the sink of $E.$ The quotient graph of $(H, \emptyset)$ is the second and the porcupine graph is the third graph below. We also note that the hedgehog graph (see \cite[Definition 2.5.16]{LPA_book}) is the fourth graph below. 
\[\xymatrix{{\bullet}\ar@(lu,ld)  \ar[r] & {\bullet}}\hskip2cm\xymatrix{{\bullet}\ar@(lu,ld)}\hskip2cm\xymatrix{\ar@{.>}[r] &
\bullet \ar[r] &  \bullet\ar[r]&  \bullet\ar[r] & 
\bullet }\hskip2cm\xymatrix{\bullet\ar[dr] &  \bullet\ar[d]  &  \circ \ar@{.>}[dl] \\
\bullet \ar[r]  & \bullet   &  \circ \ar@{.>}[l] }\]
The comparison of the porcupine and the hedgehog illustrates the point from the introduction: the hedgehog graph of an  admissible pair can have more ``spines'' and they are short (all of length one) and the porcupine graph can have fewer ``spines'' and they can be (and often are) of length larger than one. 

In addition, let $F$ be the first graph below and let $G$ be its sink. The second graph below is the quotient graph  and the third graph below is the porcupine graph of $(G,\emptyset)$. 
{\small
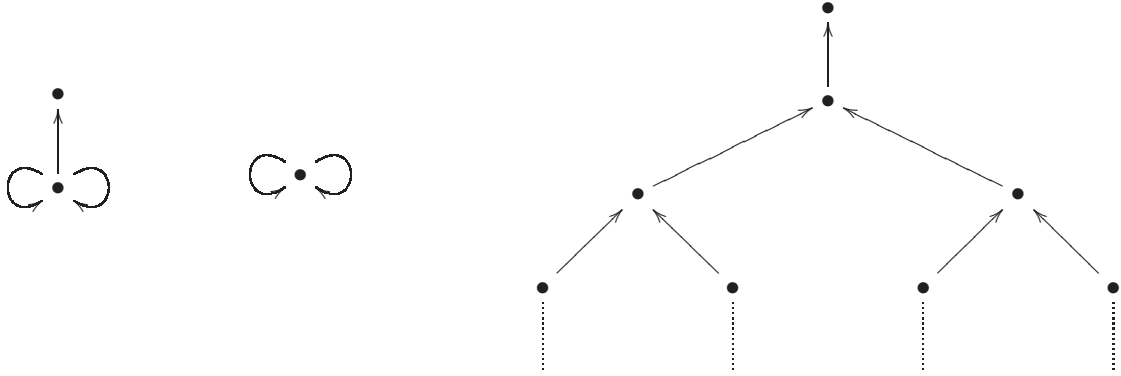
\[\xymatrix{\\\bullet\\\bullet\ar@(lu,ld)\ar@(ru, rd)  \ar[u] }\hskip2.8cm
\xymatrix{\\\\\bullet\ar@(lu,ld)\ar@(ru, rd)}\hskip2.8cm 
\xymatrix{
&&&\bullet&&&\\
&&&\bullet\ar[u]&&&\\
&\bullet\ar[urr]&&&&\bullet\ar[ull]&\\
\bullet\ar[ur]\ar@{.}[d]&&\bullet\ar[ul]\ar@{.}[d]&&\bullet\ar[ur]\ar@{.}[d]&&\bullet\ar[ul]\ar@{.}[d]\\
&&&&&&\\
}
\]}
While the hedgehog graph of $G$ is the {\em same} as the hedgehog graph of $H$ from the previous example, we can see that the two corresponding porcupine graphs are very different. This illustrates how the porcupine graph retains more information from the original graph than the hedgehog graph.  
\end{example}

We finish this subsection with an observation and a lemma. If $R$ is a ring, $J$ is its ideal which is locally unital as a ring, and $I$ an ideal of $J,$ then $I$ is an ideal of $R$. Indeed, if $x\in I$ and  $r\in R,$ there is $u\in J$ which is a local unit for $x$ so that $xr=(xu)r=(xu)(ur)\in IJ\subseteq I.$ Similarly, $rx\in I.$ By an analogous argument, if $\Gamma$ is any group, $R$ is a $\Gamma$-graded ring, $J$ is a graded ideal of $R$ which is (graded) locally unital as a ring, and if $I$ is a graded ideal of $J,$ then $I$ is a graded ideal of $R.$  

For a Leavitt path algebra, finite sums of vertices are homogeneous local units. Every (graded) ideal is (graded) isomorphic to a Leavitt path algebra by the porcupine graph construction, so it is also (graded) locally unital. Thus, the above observation proves the following lemma. 

\begin{lemma}
If $E$ is any graph and $I$ is a (graded) ideal of $L_K(E),$ then any (graded) ideal of $I$ is a (graded) ideal of $L_K(E).$
\label{lemma_graded_ideal_transitive}
\end{lemma}

\subsection{Pre-order monoids and their order-ideals}\label{subsection_order-ideals}
An abelian monoid $M$ with a reflexive and transitive relation (a pre-order) $\geq$ is a {\em pre-ordered monoid} if $x\geq y$ implies $x + z\geq y + z$ for all $x, y, z\in M.$ A submonoid $I$ of a pre-ordered monoid $M$ is an {\em order-ideal} of $M$ if $x+y\in I$ implies $x\in I$ and $y\in I$ (equivalently $x\geq y$ and $x\in I$ implies $y\in I$). 

If $\Gamma$ is a group and $M$ a pre-ordered monoid with a left action of $\Gamma,$ then $M$ is a {\em pre-ordered $\Gamma$-monoid} if $x\geq y$ implies $\gamma x\geq \gamma y$ for all $x, y\in M$ and $\gamma\in \Gamma.$ A $\Gamma$-submonoid $I$ of a pre-ordered $\Gamma$-monoid $M$ which is an order-ideal is  a {\em $\Gamma$-order-ideal}.

\subsection{The graph monoid and the talented monoid}\label{subsection_graph_monoid}
For any infinite emitter $v$ of a graph $E$ and any finite and nonempty $Z\subseteq \so^{-1}(v),$ let $q^v_Z=v-\sum_{e\in Z}ee^*.$ The {\em graph monoid} $M_E$ is the free abelian monoid on generators $[v]$ for $v\in E^0$ and $[q^v_Z]$ for infinite emitters $v$ and nonempty and finite sets $Z\subseteq \so^{-1}(v)$ subject to the relations
\[[v]=\sum_{e\in \so^{-1}(v)}[\ra(e)],\hskip.4cm [v]=[q^v_Z]+\sum_{e\in Z}[\ra(e)],\,\mbox{ and }\;[q^v_Z]=[q^v_W]+\sum_{e\in W-Z}[\ra(e)]\]
where $v$ is a vertex which is regular for the first relation and an infinite emitter for the second two relations in which $Z\subsetneq W$ are finite and nonempty subsets of $\so^{-1}(v).$
The map $[v]\mapsto [vL_K(E)]$ and $[q^v_Z]\mapsto [q^v_ZL_K(E)]$ extends to an isomorphism $\gamma_E$ of $M_E$ and $\V(L_K(E))$ by \cite[Corollary 3.2.11]{LPA_book}.   

If $\Gamma=\langle t\rangle$ is the infinite cyclic group on $t,$ the {\em talented monoid } or the {\em graph $\Gamma$-monoid $M_E^\Gamma$}  is the free abelian $\Gamma$-monoid on the same generators as $M_E$ subject to the relations
\[[v]=\sum_{e\in \so^{-1}(v)}t[\ra(e)],\hskip.4cm [v]=[q^v_Z]+\sum_{e\in Z}t[\ra(e)],\,\mbox{ and }\;[q^v_Z]=[q^v_W]+\sum_{e\in W-Z}t[\ra(e)]\]
where $v, Z,$ and $W$ have the same properties as for the defining relations of $M_E.$ While the monoid $M_E$ can register only whether two vertices are connected, the ``talent'' of $M_E^\Gamma$ is to register the lengths of paths between vertices: if $p$ is a path of length $n$, the relation $[\so(p)]=t^n[\ra(p)]+x$ holds in $M_E^\Gamma$ for some $x\in M_E^\Gamma.$ If $\V^{\Gamma}(L_K(E))$ is the monoid of the graded isomorphism classes $[P]$ of finitely generated graded projective right $R$-modules $P$ with the addition $[P]+[Q]=[P\oplus Q]$ and the left $\Gamma$-action $(\gamma, [P])\mapsto [(\gamma^{-1})P],$ then the map $[v]\mapsto [vL_K(E)]$ and $[q^v_Z]\mapsto [q^v_ZL_K(E)]$ extends to an isomorphism $\gamma_E^\Gamma$ of $M_E^\Gamma$ and $\V^{\Gamma}(L_K(E))$ (\cite[Proposition 5.7]{Ara_et_al_Steinberg}).

\section{Porcupine-quotient graph}
\label{section_porcupine_quotient}

In this section, we generalize the constructions of the quotient and the porcupine graphs by introducing the porcupine-quotient graph corresponding to the quotient of one admissible pair with respect to another admissible pair. By Theorem \ref{theorem_porcupine_quotient}, the Leavitt path algebra of this graph is graded isomorphic to the quotient of two corresponding graded ideals. 

If $H\subseteq G$ are two sets of vertices of $E$, let   
\[B_H^G=\{v\in E^0-H\mid v\mbox{ is an infinite emitter and } 0<|\so^{-1}(v)\cap \ra^{-1}(G-H)|<\infty\}.\]

\begin{definition}
If $(H,S)$ and $(G,T)$ are two admissible pairs of a graph $E$ such that $(H,S)\leq (G,T),$ we let 
\[
\begin{array}{ll}
F_1(G-H, T-S)=\{e_1e_2\ldots e_n\mbox{ is a path of }E\mid \ra(e_n)\in G-H, \so(e_n)\notin (G-H)\cup (T-S)
\}\mbox{ and}\\
F_2(G-H, T-S)=\{p\mbox{ is a path of }E\mid \ra(p)\in T-S,  |p|>0\}.
\end{array}
\]

The {\em porcupine-quotient graph} $(G,T)/(H,S)$ of $(G,T)$ with respect to $(H,S)$ is defined as follows. 
The set of vertices of $(G,T)/(H,S)$ is the  set
\[(G-H)\cup (T-S)\cup \{w^p\mid p\in F_1(G-H,T-S)\cup F_2(G-H,T-S)\}\cup \{v'\mid v\in ((G\cup T)-S)\cap B^G_H
\}.\]  
The set of edges of $(G,T)/(H,S)$ is the set 
\[\{e\in E^1\mid \ra(e)\in G-H\mbox{ and either }\so(e)\in G-H\mbox{ or }\so(e)\in T-S\}\cup\]\[\{f^p\mid p\in F_1(G-H,T-S)\cup F_2(G-H,T-S)\}\cup \{e'\mid \ra(e)\in ((G\cup T)-S)\cap B^G_H\}.\] 

The source and range of an edge of $(G,T)/(H,S)$ which is also in $E^1$ are the same as in $E.$ 

If $e\in E^1\cap(F_1(G-H, T-S)\cup F_2(G-H, T-S)),$ we let $\so(f^e)=w^e$ and $\ra(f^e)=\ra(e).$
If $p=eq$ where $e\in E^1,$ $q\in F_1(G-H, T-S)\cup F_2(G-H, T-S),$ and $|q|>0,$ let $\so(f^p)=w^p$ and $\ra(f^p)=w^q.$

If $\ra(e)\in ((G\cup T)-S)\cap B^G_H,$ we let $\ra(e')= \ra(e)'.$ If $\ra(e)\in (G-S)\cap B_H^G$ and if $\so(e)\in (G-H)\cup (T-S),$ we let $\so(e')=\so(e).$ If $\so(e)\notin (G-H)\cup (T-S),$ then either $\so(e)\notin G\cup T$ or $\so(e)\in T\cap S.$ In either case,  $e\in F_1(G-H, T-S)$ and we let $\so(e')=w^e.$ If $\ra(e)\in (T-S)\cap B_H^G,$ then $e\in F_2(G-H, T-S)$ and we let $\so(e')=w^e.$   

If $S=T=\emptyset,$ we write $(G, \emptyset)/(H,\emptyset)$ shorter as $G/H.$
\label{definition_porcupine_quotient} 
\end{definition}

If $(G,T)=(E^0, \emptyset),$ the porcupine-quotient graph is exactly the quotient graph $E/(H,S)$ since $F_1(E^0-H, \emptyset)=F_2(E^0-H, \emptyset)=\emptyset,$  $B^{E^0}_H=B_H,$ and $(E^0-S)\cap B_H=B_H-S$ so the added vertices and edges are exactly as in $E/(H,S).$ 

If $(H,S)=(\emptyset, \emptyset),$ the porcupine-quotient graph is exactly the porcupine graph $P_{(G,T)}$ since $ F_1(G-\emptyset,T-\emptyset)= F_1(G,T),$ $F_2(G-\emptyset,T-\emptyset)= F_2(G,T),$ and $(G\cup T)\cap B^G_{\emptyset}=\emptyset$ 
because if $v$ is an infinite emitter in $G\cup T,$ then $v$ emits infinitely many edges to $G,$ so $v$ is not in  $B^G_{\emptyset}.$ 

We present some examples illustrating the construction. 

\begin{example}
\begin{enumerate}
\item Let $E$ be the graph $\xymatrix{   
\bullet_{u_0}  & \bullet_{v_0}  & \bullet_{w_0}  \\   
\bullet_{u_1}\ar[r]^e\ar[u] & \bullet_{v_1}\ar[r]^g\ar[u]_h & \bullet_{w_1}\ar[u] }$
and let $H=\{w_0, w_1\}$ and $G=H\cup\{v_0, v_1\}.$ 
Then, $G/H$ is the graph $\xymatrix{   
& \bullet_{v_0}    \\   
\bullet\ar[r]^{f^e} & \bullet_{v_1}\ar[u]_h }.$ 
The quotient $I(G)/I(H)$ 
is generated, as a graded $*$-algebra, by three elements of degree zero, $v_0+I(H),$ $v_1+I(H),$ and $ee^*+I(H),$ two elements of degree one, $e+I(H)$ and $h+I(H),$ and one element of degree two, $eh+I(H).$  
The Leavitt path algebra of the porcupine-quotient graph is generated by three elements of degree zero, $v_0,$ $v_1,$ and $f^e(f^e)^*,$ two elements of degree one, $f^e$ and $h,$ and the path $f^eh$ of degree two. The correspondence mapping the generators of $L_K(G/H)$ to the generators of $I(G)/I(H)$ in the order listed above extends to a graded $*$-homomorphism $L_K(G/H)\to I(G)/I(H)$ (this also follows from the proof of Theorem \ref{theorem_porcupine_quotient}).  

The porcupine graph of $H$ is 
$
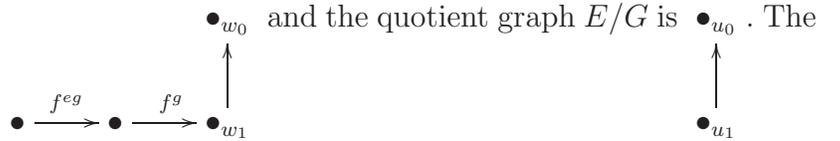
\xymatrix{   
&& \bullet_{w_0}    \\   
\bullet\ar[r]^{f^{eg}}&\bullet\ar[r]^{f^g} & \bullet_{w_1}\ar[u]  }$ and the quotient graph $E/G$ is 
$\xymatrix{   
\bullet_{u_0}  \\   
\bullet_{u_1}\ar[u]}.$ The chain $\emptyset\leq H\leq G\leq E^0$ is such that the porcupine-quotient graph of each two consecutive terms is cofinal. 

\item Let $E$ be the graph $\xymatrix{\bullet_v\ar@(lu,ld)_e
\ar@{.} @/_1pc/ [r] _{\mbox{ } } \ar@/_/ [r] \ar [r] \ar@/^/ [r] \ar@/^1pc/ [r] &\bullet_w}.$ For $H=\{w\}$ and $B_H=\{v\},$ $(H, \{v\})/(H, \emptyset)$ is the graph $\xymatrix{ \ar@{.}[r]&\bullet\ar[r]^{f^{eee}}&\bullet\ar[r]^{f^{ee}}&\bullet\ar[r]^{f^{e}}& \bullet_v}.$ If $g_1,g_2,\ldots$ are the edges $v$ emits to $w,$ the porcupine graph of $(H,\emptyset)$ is 
$
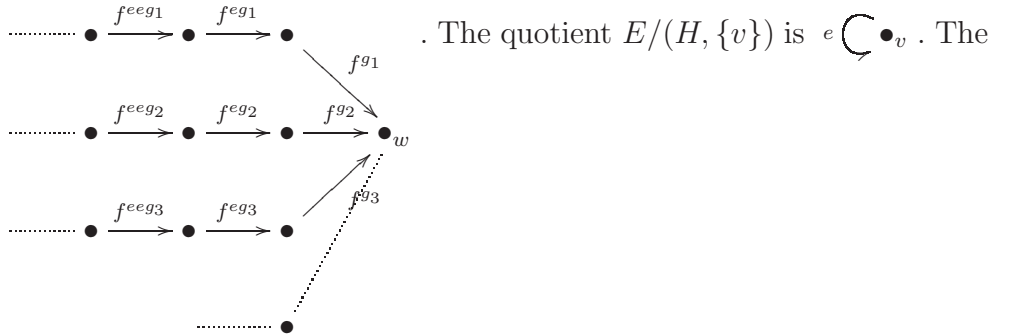
\xymatrix{ \ar@{.}[r]&\bullet\ar[r]^{f^{eeg_1}}&\bullet\ar[r]^{f^{eg_1}}&\bullet\ar[dr]^{f^{g_1}}&\\\ar@{.}[r]&\bullet\ar[r]^{f^{eeg_2}}&\bullet\ar[r]^{f^{eg_2}}&\bullet\ar[r]^{f^{g_2}}& \bullet_w\\\ar@{.}[r]&\bullet\ar[r]^{f^{eeg_3}}&\bullet\ar[r]^{f^{eg_3}}&\bullet\ar[ur]_{f^{g_3}}&\\&&\ar@{.}[r]&\bullet\ar@{.}[uur]&}.$
The quotient $E/(H,\{v\})$ is $\xymatrix{\bullet_v\ar@(lu,ld)_e}.$ The chain $(\emptyset,\emptyset)\leq(H, \emptyset)\leq(H, \{v\})\leq(E^0, \emptyset)$ is such that the porcupine-quotient graph of each two consecutive terms is cofinal. 
\end{enumerate}
\label{example_three_porcupine_quotients}
\end{example}

The following example generalizes the last example and exhibits a scenario appearing in the proof of Theorem \ref{theorem_comp_series}. 

\begin{example}
Let $E$ be any graph and $H$ be a hereditary and saturated set with $B_H$ nonempty. Let $S\subsetneq S\cup \{v\}\subseteq B_H.$ We describe the porcupine-quotient $(H, S\cup \{v\})/(H, S).$ As $B_H^H=\emptyset,$ no vertices of the form $v'$ are present. We also have that $F_1(\emptyset, \{v\})=\emptyset,$ so the only vertices of this graph beside $v$ are the vertices of the form $w^p$ for $p\in F_2(\emptyset, \{v\}).$ The vertex $v$ is a sink and each vertex of the form $w^p$ emits only one edge. For each $p\in  F_2(\emptyset, \{v\})$ there is only one path from $w^p$ to $v$ and there are neither cycles, infinite emitters, nor infinite paths in this graph. By Lemma \ref{lemma_saturated_closure}, condition (3a) of Theorem \ref{theorem_graded_simple} holds. Hence,  $(H, S\cup \{v\})/(H, S)$ is cofinal. 
\label{example_with_B_H_and_emptyset} 
\end{example}

\begin{remark} {\bf The porcupine-quotient graph versus the relative quotient graph.}
If $E$ is any graph and $H$ and $G$ are two hereditary and saturated sets of vertices such that $H\subseteq G$, the authors of \cite{Roozbeh_Alfilgen_Jocelyn} define the quotient $Q$ of $G$ with respect to $H$ as the graph with  
\[Q^0=G-H\mbox{ and }Q^1=\{e\in E^1\mid \so(e)\in G, \ra(e)\notin H\}\]
and $\so$ and $\ra$ relations the same as in $E.$ This construction is different than the porcupine-quotient $G/H,$ so we refer to it as the {\em relative quotient}. Even if $E$ is row-finite, the two constructions are different since the vertices of the porcupine-quotient $G/H$ are not only the vertices of $G-H$ but also the vertices of the form $w^p$ for $p\in F_1(G-H)=\{p=e_1\ldots e_n\mid \ra(e_n)\in G-H, \so(e_n)\notin G\}.$ For example, if $E$ is the graph from part (1) of Example \ref{example_three_porcupine_quotients} and $H$ and $G$ as in that example, then the porcupine-quotient graph is $
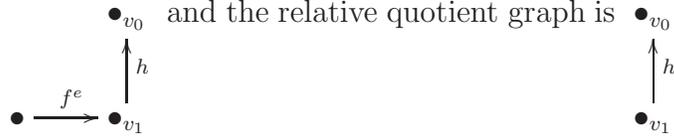
\xymatrix{   
& \bullet_{v_0}    \\   
\bullet\ar[r]^{f^e} & \bullet_{v_1}\ar[u]_h }$ and the relative quotient graph is $\xymatrix{   
\bullet_{v_0}    \\   
\bullet_{v_1}\ar[u]_h}.$ While the porcupine-quotient graph retains the information on the number of paths ending in $G-H$, this information is lost in the relative quotient graph. By Example \ref{example_three_porcupine_quotients} (and, also, by Theorem \ref{theorem_porcupine_quotient}), the Leavitt path algebra of the porcupine-quotient is graded isomorphic to the quotient $I(G)/I(H).$ We claim that the Leavitt path algebra of the relative quotient is not isomorphic to $I(G)/I(H).$ 

The quotient $I(G)/I(H)$ 
is generated, as a graded $*$-algebra, by the six elements listed in Example \ref{example_three_porcupine_quotients}. As a $*$-algebra, it is $*$-isomorphic to $\M_3(K)$. On the other hand, the Leavitt path algebra of the relative quotient has only three generators $v_0, v_1,$ and $h$ as a $*$-algebra and it is isomorphic to $\M_2(K)$. The algebras $\M_2(K)$ and $\M_3(K)$ are not isomorphic. 

In \cite{Roozbeh_Alfilgen_Jocelyn}, the  relative quotients are considered only as the underlying graphs of their talented monoids. The talented monoids of the relative and the porcupine-quotients are isomorphic, so both constructions can be used. However, when the talented monoids are considered with their order-units, the constructions are different. Example \ref{example_composition_talented} contains more details of this last point.  
\label{remark_relative_quotient}
\end{remark}

Before proving Theorem \ref{theorem_porcupine_quotient}, we prove an auxiliary lemma which generalizes \cite[Theorem 2.4.8]{LPA_book}.
\begin{lemma} If $E$ is any graph, $(H,S)$ and $(G,T)$ are admissible pairs such that $(H,S)\leq (G,T),$ and $v\in B_G,$ then 
$v^G\in I(H,S)$ if and only if $v
\in S$ and $v$ does not emit any edges to $G-H.$ 
\label{lemma_v^G} 
\end{lemma}
\begin{proof}
To show the implication $(\Rightarrow),$ assume that $v^G\in I(H,S).$ If $v$ does not emit any edges to $G-H,$ then $v\in B_H$ and $v^H=v^G\in I(H,S)$ which implies that $v\in S$ by \cite[Theorem 2.4.8]{LPA_book}. If $v$ emits some edges to $G-H,$ let $e$ be one of them.  
As $v^G\in I(H,S),$  $e^*v^Ge=e^*e=\ra(e)\in I(H,S).$ By \cite[Theorem 2.4.8]{LPA_book}, $\ra(e)\in H$ which is a contradiction because $\ra(e)\in G-H.$ 

The converse $(\Leftarrow)$ 
holds since $v\in S$ implies that $v^H\in I(H,S)$ by \cite[Theorem 2.4.8]{LPA_book} and $\so^{-1}(v)\cap \ra^{-1}(G-H)=\emptyset$ implies that $v^G=v^H.$ 
\end{proof}

\begin{theorem} If $(H,S)$ and $(G,T)$ admissible pairs of a graph $E$ such that $(H,S)\leq (G,T),$ then the algebras 
$L_K((G,T)/(H,S))$ and $I(G,T)/I(H,S)$ are graded isomorphic. 
\label{theorem_porcupine_quotient}
\end{theorem}
\begin{proof}
To shorten the notation in the proof, we let $I=I(H,S),$ $E_v=\so^{-1}(v)\cap \ra^{-1}(G-H)$ for any $v\in E^0,$ and, if $E_v$ is finite and nonempty, we let 
\[v^{G-H}=\sum_{e\in E_v}ee^*.\]
We define a map $\phi:L_K((G,T)/(H,S))\to I(G,T)/I$ by mapping the vertices of $(G,T)/(H,S)$ as follows. 
\begin{center}
\begin{tabular}{lllll}
$v$&$\mapsto$&$ v$ & $+I\;\;$ & if $v\in (G-H)- B_H^G\cup (G\cap S),$\\
$v$&$\mapsto$&$v^{G-H}$ & $ +I\;\;$ & if $v\in ((G\cup T)-S))\cap B_H^G$ \\
$v$&$\mapsto$&$ v^G$ & $ +I\;\;$ & if $v\in (T-S)-B_H^G,$\\
$w^p$&$\mapsto$&$ pp^*$ & $+I$ 
& if $p\in F_1(G-H,T-S),$\\
$w^p$&$\mapsto$&$ p\ra(p)^Gp^*$ & $+I$ 
& if $p\in F_2(G-H,T-S),$\\
$v'$&$\mapsto$&$ v-v^{G-H}$ & $+I$ & if $v\in (G-S)\cap B_H^G,$  \\  
$v'$&$\mapsto$&$ v^G-v^{G-H}$ & $+I$ & if $v\in (T-S)\cap B_H^G.$  \\  
\end{tabular}
\end{center}
One directly checks that the union of the sets in the if-parts of the first three cases is indeed $(G-H)\cup (T-S).$ Note also that
$v\in (G-S)\cap B_H^G,$ then $v\in B_H$ and $v$ does not emit edges outside of $G,$ so $v-v^{G-H}=v^H.$ If $v\in (T-S)\cap B_H^G,$ then $v\in B_H$ and $v^G-v^{G-H}=v^H.$ Thus, 
\[\phi(v')=v^H+I\]
for any $v\in ((G\cup T)-S)\cap B_H^G$ and the last two lines of the above definition can be condensed into one. While the longer, ``non-condensed'' definition of $\phi(v')$ increases clarity of some parts of the following proof, we occasionally use also the ``condensed'' version. 

We define $\phi$ on the edges of $(G,T)/(H,S)$ by  
\begin{center}
\begin{tabular}{llll}
$e$&$\mapsto$&$(e+I)\phi(\ra(e))$ & if $e\in E^1,$ \\
$f^p$&$\mapsto$&$ (e+I)\phi(\ra(f^p))$ & if $p=eq\in F_1(G-H, T-S)\cup F_2(G-H,T-S),$  $e\in E^1,$\\
$e'$&$\mapsto$&$(e+I)\phi(\ra(e)')$ & if $\ra(e)\in ((G\cup T)-S)\cap B_H^G,$ \\ 
\end{tabular}
\end{center}
and we define $\phi$ on the set of ghost edges by  $\phi(g^*)=\phi(g)^*$ for any edge $g$ of the graph $(G,T)/(H,S).$ 

{\bf Extending $\mathbf\phi$ to a graded $*$-homomorphism.} 
One directly checks that the axioms (V) and (CK1) hold and that the part of (E1) involving the range function holds for the images of the vertices and edges of the porcupine-quotient graph. If $e$ is an edge of both the porcupine-quotient graph and of $E$, one checks that $\phi(\so(e))(e+I)=e+I$ so that $\phi(\so(e))\phi(e)=\phi(\so(e))(e+I)\phi(\ra(e))=(e+I)\phi(\ra(e))=\phi(e).$

Let $p=eq\in F_1(G-H, T-S)$ for an edge $e$ and a path $q$. If $|q|>0,$ then  
$\phi(\so(f^p))\phi(f^p)=
\phi(w^p)(e+I)\phi(w^q)=
eqq^*e^*eqq^*+I=eqq^*+I=(e+I)(qq^*+I)= (e+I)\phi(w^q)=\phi(f^p).$ If $|q|=0,$ then 
$\phi(\so(f^e))\phi(f^e)=
\phi(w^e)(e+I)\phi(\ra(e))=
(ee^*e+I)\phi(\ra(e))=(e+I)\phi(\ra(e))= \phi(f^e).$ 
Checking that $\phi(\so(f^p))\phi(f^p)=\phi(f^p)$ for $p\in F_2(G-H, T-S)$ is similar. 

If $e'$ is defined and if $\so(e')=\so(e),$ then  
\[\phi(\so(e'))\phi(e')=\phi(\so(e))(e+I)\phi(\ra(e)')=(e+I)\phi(\ra(e)')=\phi(e').\]
If $e'$ is defined and if $e\in F_1(G-H, T-S),$ so that $\so(e')=w^e,$   
then 
\[\phi(\so(e'))\phi(e')=(ee^*+I)(e+I)\phi(\ra(e)')=(e+I)\phi(\ra(e)')=\phi(e').\]
If $e'$ is defined and if  $e\in F_2(G-H, T-S)$ so that $\so(e')=w^e,$ 
then 
\[\phi(\so(e'))\phi(e')=e\ra(e)^Ge^*e\ra(e)^H+I=e\ra(e)^G\ra(e)^H+I=e\ra(e)^H+I=
(e+I)\phi(\ra(e)')=\phi(e').\]

This shows that (E1) holds. By the definition of $\phi$ on the ghost edges, (E1) holding implies that (E2) also holds. So, it remains to check (CK2). 

If $v\in E^0$ is a regular vertex of $(G,T)/(H,S),$ then either $v$ is a regular vertex of $E$ which is in $G-H$ (hence it does not emit all of its edges to $H$), or $v\in G\cap S,$ or $v\in ((G\cup T)-S)\cap B_H^G.$ In any case, the set $E_v$ is nonempty and finite. Let us partition $E_v$ into two sets, $E_{v1}=E_v\cap \ra^{-1}((G-S)\cap B_H^G)$ and $E_{v2}=E_v-E_{v1}.$ Note that any of these two sets can possibly be empty, but not both. If any of them is empty, let $0$ stands for $\sum_{e\in \emptyset}ee^*.$ In $(G,T)/(H,S),$ $v$ emits the edges $e\in E_v$ and $e'$ for $e\in E_{v1}$ and we have that
\[\sum_{e\in E_v}\phi(e)\phi(e^*)+\sum_{e\in E_{v1}}\phi(e')\phi((e')^*)=
\sum_{e\in E_v}(e+I)\phi (\ra(e))(e^*+I)+\sum_{e\in E_{v1}}(e+I)\phi(\ra(e)')(e^*+I)=\]\[\sum_{e\in E_{v1}}(e\ra(e)^{G-H}e^*+I)+
\sum_{e\in E_{v2}}(e\ra(e)e^*+I)+\sum_{e\in E_{v1}}(e(\ra(e)-\ra(e)^{G-H})e^*+I)=\] 
\[\left(\sum_{e\in E_{v1}} e\ra(e)^{G-H}e^*+\sum_{e\in E_{v2}}ee^*+\sum_{e\in E_{v1}}ee^*-\sum_{e\in E_{v1}}e\ra(e)^{G-H}e^*\right)+I=\sum_{e\in E_v}ee^*+I=v^{G-H}+I.\]

It remains to show that $\phi(v)=v^{G-H}+I$ in any of the three possibilities for $v$. 

If $v$ is a regular vertex of $E$ which is in $G-H,$ then $ee^*\in I$ for every $e\in \so^{-1}(v)\cap \ra^{-1}(H),$ so 
\[\phi(v)=v+I=\sum_{e\in \so^{-1}(v)}ee^*+I=\sum_{e\in E_v}ee^*+I=v^{G-H}+I.\]

If $v\in G\cap S,$ then $v$ does not emit any edges outside of $G$ so $v-v^{G-H}=v^H\in I.$ Thus, $\phi(v)=v+I=v^{G-H}+I.$  

If $v\in ((G\cup T)-S)\cap B_H^G,$ $\phi(v)=v^{G-H}+I$ by the definition of $\phi.$  

As the vertices of the form $v'$ are not regular in the porcupine-quotient, it remains to check (CK2) for the vertices of the form $w^p$ for $p\in F_1(G-H,T-S)\cup F_2(G-H,T-S).$ If $p=eq$ for $e\in E^1$ and $|q|>0,$ then $w^p$
emits only $f^p,$ and \[\phi(f^p)\phi((f^p)^*)=(e+I)\phi(w^q)(e^*+I)=eqq^*e^*+I=pp^*+I=\phi(w^p)\] for  $p\in F_1(G-H,T-S)$ and \[\phi(f^p)\phi((f^p)^*)=(e+I)\phi(w^q)(e^*+I)=eq\ra(p)^Gq^*e^*+I=p\ra(p)^Gp^*+I=\phi(w^p)\] for $p\in F_2(G-H,T-S).$

If $p=e\in F_1(G-H,T-S)$ and if $\ra(e)\in (G-S)\cap B_H^G$ then $w^e$ emits two edges, $f^e$ and $e',$ and 
\[\phi(f^e)\phi((f^e)^*)+\phi(e')\phi((e')^*)=e\ra(e)^{G-H}e^*+e(\ra(e)-\ra(e)^{G-H})e^*+I=ee^*+I=\phi(w^e).\]

If $p=e\in F_1(G-H,T-S)$ and $\ra(e)\notin (G-S)\cap B_H^G,$ then $w^e$ emits only $f^e$ and \[\phi(f^e)\phi((f^e)^*)=ee^*+I=\phi(w^e).\]

If $p=e\in F_2(G-H,T-S)$ and $\ra(e)\in (T-S)\cap B_H^G,$ then $w^e$ emits $f^e$ and $e'$ and 
\[\phi(f^e)\phi((f^e)^*)+\phi(e')\phi((e')^*)=e\ra(e)^{G-H}e^*+e(\ra(e)^G-\ra(e)^{G-H})e^*+I=e\ra(e)^Ge^*+I=\phi(w^e).\]
If $p=e\in F_2(G-H,T-S)$ and $\ra(e)\notin (T-S)\cap B_H^G,$ then $\phi(f^e)\phi((f^e)^*)=e\ra(e)^Ge^*+I=\phi(w^e).$

This shows that all five axioms hold for the images of vertices, edges and ghost edges of $(G,T)/(H,S).$ By the Universal Property, $\phi$ extends to a unique homomorphism, which we denote also by $\phi,$ of $L_K((G,T)/(H,S)).$ The map $\phi$ is a $*$-homomorphism since the images of vertices are selfadjoint and by the definition of $\phi$ on the ghost edges. The map $\phi$ is graded because the vertices are mapped to the elements of degree zero and the edges to the elements of degree one.

{\bf Showing injectivity.}
To use the Graded Uniqueness Theorem and conclude that $\phi$ is injective, we need to check that the images of the vertices are not in $I.$ This is clear for the vertices in $(G-H)-B^G_H\cup (G\cap S)$ because they are in $E^0-H,$ so they are not elements of $I.$ By Lemma \ref{lemma_v^G}, $\phi(v)=v^G+I\neq I$ for $v\in (T-S)-B_H^G$ and $\phi(v')=v^H+I\neq I$ as $v\notin S$ for  $v\in ((G\cup T)-S)\cap B_H^G.$ 

If $v\in ((G\cup T)-S)\cap B_H^G,$ assuming that $\phi(v)=v^{G-H}+I=I$ implies that $\ra(e)=e^*ee^*e=e^*v^{G-H}e\in I$ for any $e\in E_v.$ This is a contradiction since $\ra(e)\notin H$ by the definition of $E_v.$ 

Assuming that $pp^*\in I$ for some $p\in F_1(G-H, T-S)$ implies that  $\ra(p)=p^*pp^*p\in I$ which is a contradiction as $\ra(p)\in G-H.$ Similarly,
assuming that $p\ra(p)^Gp^*\in I$ for some $p\in F_2(G-H, T-S)$ implies that $\ra(p)^G=p^*p\ra(p)^Gp^*p\in I$ which is a contradiction by Lemma \ref{lemma_v^G} as $\ra(p)\in T-S.$

{\bf Showing surjectivity.}
As $\phi$ is a $*$-homomorphism, to show surjectivity of $\phi$, it is sufficient to show that $p+I$ is in the image of $\phi$ for every path $p$ such that $\ra(p)\in G$, and that $p\ra(p)^G+I$ is in the image of $\phi$ for every path $p$ such that $\ra(p)\in T$. We refer to these conditions as cases 1 and 2. 

{\bf Case 1 for paths of zero length.}
For $p=v\in G,$ we consider two cases:  
$v\in ((G-S)-B_H^G)\cup (G\cap S)$ and $v\in (G-S)\cap B_H^G.$ In the first case, $\phi(v)=v+I$ if $v\notin H$ and $\phi(0)=I=v+I$ if $v\in H.$
In the second case, $\phi(v+v')=v^{G-H}+v-v^{G-H}+I=v+I.$ 
 
{\bf Case 2 for paths of zero length.} 
If $p=v\in T,$ we consider three cases: $v\in (T-S)-B^G_H,$ $v\in (T-S)\cap B^G_H,$ and $v\in S\cap T.$ In the first case, $\phi(v)=v^G+I$ by the definition of $\phi.$ In the second case, $\phi(v+v')=v^{G-H}+v^G-v^{G-H}+I=v^G+I.$

If $v\in S\cap T$, then
either $v^H=v^G$ and $v\notin B_H^G$ or $v\in B_H^G.$ In the first case, $\phi(0)=I=v^H+I=v^G+I.$ In the second case, 
the set $E_v$ is nonempty and finite, $v^G-v^{G-H}=v^H\in I,$ and $e$ is in $F_1(G-H, T-S)$ for every $e\in E_v.$ Thus,  
\[\phi\left(\sum_{e\in E_v}w^e\right)=\sum_{e\in E_v}\phi(w^e)=\sum_{e\in E_v}ee^*+I=v^{G-H}+I=
v^G+I.\] 

{\bf Case 1 for paths of positive lengths.}
If $\ra(p)\in H$, then $p\in I$ and $\phi(0)=p+I.$ So, consider a path $p$ with $\ra(p)\in G-H.$ Let $p = eq$ with $e\in E^1,$ and either $q=e_1e_2\ldots e_n$ for some $n\geq 1$ or $q=\ra(e).$ We use induction and  assume that $q+I=\phi(x)$ for some $x\in L_K((G,T)/(H,S)).$

As $\ra(p)\in G-H,$ let us consider whether there is a prefix (possibly improper) of $p$ which is in $F_1(G-H, T-S)$ or whether every prefix of $p$ is not in $F_1(G-H, T-S).$ In the first case, we consider the cases when the prefix in $F_1(G-H, T-S)$ is $ee_1\ldots e_i$ for some $i\leq n$ and when the prefix is $e$ (in which case $|q|$ is possibly zero). 

If $ee_1\ldots e_i\in F_1(G-H, T-S),$ for some $i\leq n,$ then
\[\phi(f^{ee_1\ldots e_i}x)=(e+I)\phi(w^{e_1\ldots e_i})\phi(x)=ee_1\ldots e_i(e_1\ldots e_i)^*q+I=ee_1\ldots e_ie_{i+1}\ldots e_n+I=eq+I=p+I.\]

If $e\in F_1(G-H, T-S),$ we check whether 
$\ra(e)\in (G-S)-B_H^G\cup(G\cap S),$ $\ra(e)\in (G-S)\cap B_H^G$ and $q$ has positive length, or $\ra(e)\in (G-S)\cap B_H^G$ and $q$ has zero length. 
In the first case,
\[\phi(f^e)\phi(x)=(e+I)\phi(\ra(e))(q+I)=e\ra(e)q+I=p+I.\]
In the second case, note that $\ra(e_1)\in G-H$ as $\ra(e)=\so(e_1)\in G-H$ and $G$ is hereditary. Hence,  
\[\phi(f^e)\phi(x)=(e+I)\phi(\ra(e))(q+I)=e\ra(e)^{G-H}q+I=ee_1e_1^*e_1e_2\ldots e_n+I=eq+I=p+I.\]
In the third case,
\[\phi(e)+\phi(e')=(e+I)\phi(\ra(e))+(e+I)\phi(\ra(e)')=(e+I)(\ra(e)^{G-H}+\ra(e)-\ra(e)^{G-H}+I)=e\ra(e)+I=e+I.\]

If each prefix of $p$ is not in $F_1(G-H, T-S),$ then either $\ra(e)\in G-H$ and $\so(e)\in (G-H)\cup (T-S),$ or $\ra(e)\notin G-H,$ $\ra(e_i)\in G-H,$ and $\so(e_i)\in T-S$ for some $i\geq 1$ in which case we let $i$ be the largest such $i\leq n.$ 

In the first case, if $\ra(e)\notin (G-S)\cap B_H^G,$ then $\phi(e)=(e+I)\phi(\ra(e))=(e+I)(\ra(e)+I)=e+I$ so $\phi(ex)=eq+I=p+I.$ If $\ra(e)\in (G-S)\cap B_H^G,$ then $\phi(e+e')=e\ra(e)^{G-H}+e\ra(e)-e\ra(e)^{G-H}+I=e+I$ so $\phi((e+e')x)=eq+I=p+I.$

In the second case, if $i>1,$ 
then $ee_1\ldots e_{i-1}\in F_2(G-H, T-S)$ and  
\[\phi(f^{ee_1\ldots e_{i-1}}x)=
ee_1\ldots e_{i-1}\so(e_i)^Ge_{i-1}^*\ldots e_1^*e_1\ldots e_n+I=\]\[ee_1\ldots e_{i-1}\so(e_i)^Ge_i\ldots e_n+I=ee_1\ldots e_{i-1}e_i\ldots e_n+I=eq+I=p+I\] 
where $\so(e_i)^Ge_i=\so(e_i)e_i=e_i$ because $e_i$ has its range in $G.$ Similarly, if $i=1,$ then  $\phi(f^ex)=e\ra(e)^Gq+I=eq+I=p+I.$

{\bf Case 2 for paths of positive lengths.}
For $p = eq$ with $\ra(p)\in T, e\in E^1$ and $q$ a path of $E$, we also use induction, so let us assume that $\phi(x)=q+I$ for some $x\in L_K((G,T)/(H,S)).$

If $\ra(p)\in T-S,$ then $p\in F_2(G-H, T-S)$ and \[\phi(f^{eq}x)=(e+I)\phi(w^q)(q+I)=eq\ra(p)^Gq^*q+I=p\ra(p)^G+I.\]

If $\ra(p)\in S\cap T,$ then either $\ra(p)$ emits no edges to $G-H$ and $\ra(p)^G=\ra(p)^H\in I$ so that $\phi(0)=p\ra(p)^G+I,$ or $\ra(p)$ emits nonzero and finitely many edges to $G-H$ and  
$pg\in F_1(G-H, T-S)$ for every $g\in E_{\ra(p)}.$ As $\ra(p)^G-\ra(p)^{G-H}=\ra(p)^H\in I,$ we have that 
\[\phi\left(\sum_{g\in E_{\ra(p)}}f^{pg}x\right)=\sum_{g\in E_{\ra(p)}}(e+I)\phi(w^{qg})(q+I)=\sum_{g\in E_{\ra(p)}}eqgg^*q^*q+I=eq\ra(p)^{G-H}+I=p\ra(p)^G+I.\]
This shows that $\phi$ is surjective, and concludes the proof. 
\end{proof}

\subsection{The graph monoid and the talented monoid of a porcupine-quotient graph}

In this section (as well as in sections \ref{section_composition_series} and \ref{section_talented}), $\Gamma$ is the infinite cyclic group on a generator $t.$ By \cite[Theorems 3.6.23 and 2.5.8]{LPA_book} and \cite[Theorem 5.11]{Ara_et_al_Steinberg} the following four lattices are isomorphic.
\begin{center}
\begin{tabular}{ll}
(1) The lattice of admissible pairs of $E$.& (2) The lattice of graded ideals of $L_K(E).$\\
(3) The lattice of order-ideals of $M_E.$ & (4) The lattice of $\Gamma$-order-ideals of $M_E^\Gamma.$  
\end{tabular}
\end{center}
We recall these isomorphisms. If $(H,S)$ is an admissible pair of a graph $E,$ let $I(H,S)$ be the graded ideal of $L_K(E)$ generated by $\{v\mid v\in H\}\cup\{v^H\mid v\in S\},$ let $J(H,S)$ be the order-ideal of $M_E$ generated by $\{[v]\mid v\in H\}\cup\{[v^H]\mid v\in S\},$ and let $J^\Gamma(H,S)$ be the $\Gamma$-order-ideal  of $M_E^\Gamma$ generated by the same elements as $J(H,S)$. The element $(H,S)$ of the first lattice corresponds to the elements $I(H,S), J(H,S),$ and  $J^\Gamma(H,S)$ of the second, the third and the fourth lattice, respectively. 

The natural isomorphism $\gamma_E: M_E\to \V(L_K(E))$ (see section \ref{subsection_graph_monoid}) maps the generators of $J(H,S)$ to the elements which generate $\V(I(H,S)).$ So, the restriction of $\gamma_E$ to $J(H,S),$ mapping $J(H,S)$ to $\V(I(H,S)),$ is onto. Hence, this restriction is an isomorphism. The same argument applies in the graded case and so the restriction of $\gamma^\Gamma_E: M_E^\Gamma\to \V^\Gamma(L_K(E))$ to $J^\Gamma(H,S)$ is an isomorphism of $J^\Gamma(H,S)$ and $\V^\Gamma(I(H,S)).$ 

By \cite[Proposition 3.6.17]{LPA_book} (formulated for any ring generated by idempotents), there is a canonical injective homomorphism
$\omega:\V(L_K(E))/\V(I(H,S))\to\V(L_K(E)/I(H,S))$ such that  
$[u]+\V(I(H,S))\mapsto [u+I(H,S)]$ for an idempotent $u$ of $L_K(E).$ We review the argument from the proof of  \cite[Theorem 3.6.23]{LPA_book} showing that $\omega$ is onto. For $v\in E^0-H$ and $Z\subseteq \so^{-1}(v)\cap \ra^{-1}(E^0-H)$ finite but possibly empty, the elements of the form $[v-\sum_{e\in Z} ee^*+I(H,S)]$ generate
$\V(L_K(E)/I(H,S)).$ As such elements are in the image of $\omega,$ $\omega$ is onto.  

It is direct to check that \cite[Proposition 3.6.17]{LPA_book} holds for $\Gamma$-graded rings generated by homogeneous idempotents and so there is an  injective homomorphism
$\omega^\Gamma: \V^\Gamma(L_K(E))/\V^\Gamma(I(H,S))\to\V^\Gamma(L_K(E)/I(H,S))$ of pre-ordered $\Gamma$-monoids mapping $[u]+\V^\Gamma(I(H,S))\mapsto [u+I(H,S)]$ for a homogeneous idempotent $u$ of $L_K(E).$ The same argument for showing that $\omega$ is onto applies to $\omega^\Gamma,$ so $\omega^\Gamma$ is an isomorphism.

We use similar arguments to show the proposition below. We use the above definitions of $J(H,S)$ and $J^\Gamma(H,S)$ for an admissible pair $(H,S)$ in the statement of the proposition.  

\begin{proposition}
If $(H,S)$ and $(G,T)$ admissible pairs of a graph $E$ such that $(H,S)\leq (G,T),$ then there is a pre-ordered monoid isomorphism $M_{(G,T)/(H,S)}\cong J(G,T)/J(H,S)$ and a pre-ordered $\Gamma$-monoid isomorphism $M^\Gamma_{(G,T)/(H,S)}\cong J^\Gamma(G,T)/J^\Gamma(H,S).$   
\label{proposition_monoids_of_porcupine_quotient}
\end{proposition}
\begin{proof}
We have that $M_{(G,T)/(H,S)}\cong \V(L_K((G,T)/(H,S)))\cong \V(I(G,T)/I(H,S))$ 
where the first isomorphism is $\gamma_{(G,T)/(H,S))}$ and the second is induced by the isomorphism from Theorem \ref{theorem_porcupine_quotient}. 
By \cite[Proposition 3.6.17]{LPA_book}, there is a canonical injective homomorphism
$\omega:\V(I(G,T))/\V(I(H,S))\to \V(I(G,T)/I(H,S))$ such that $[u]+\V(I(H,S))\mapsto [u+I(H,S)]$ for any idempotent $u$ of $I(G,T).$ The elements of the form $[v-\sum_{e\in Z} ee^*+I(H,S)],$ where $v\in G-H$ and $Z\subseteq \so^{-1}(v)\cap \ra^{-1}(G-H)$ is finite but possibly empty, generate $\V(I(G,T)/I(H,S))$ and such elements are in the image of $\omega.$ Thus, $\omega$ is onto. Lastly, $\V(I(G,T))/\V(I(H,S))\cong J(G,T)/J(H,S)$ since the restrictions of $\gamma_E$ to $J(G,T)$ and $J(H,S)$ respectively, are isomorphisms $J(G,T)\to \V(I(G,T))$ and $J(H,S)\to \V(I(H,S)).$

The argument for the $\Gamma$-monoids is completely analogous. We use the $\Gamma$-monoid version of \cite[Proposition 3.6.17]{LPA_book} to obtain an injective homomorphism $\omega^\Gamma: \V^\Gamma(I(G,T))/\V^\Gamma(I(H,S))\to \V^\Gamma(I(G,T)/I(H,S))$ such that $[u]+\V^\Gamma(I(H,S))\mapsto [u+I(H,S)]$ for any homogeneous idempotent $u$ of $I(G,T).$
The map $\omega^\Gamma$ is onto by the same argument as for $\omega.$ Thus, we have the isomorphisms 
\[M^\Gamma_{(G,T)/(H,S)}\cong \V^\Gamma(L_K((G,T)/(H,S)))\cong \V^\Gamma(I(G,T)/I(H,S))\cong\]\[ \V^\Gamma(I(G,T))/\V^\Gamma(I(H,S))\cong
J^\Gamma(G,T)/J^\Gamma(H,S)
\]
where the first one is $\gamma_{(G,T)/(H,S)},$ the existence of the second follows from Theorem \ref{theorem_porcupine_quotient}, the third is the inverse of $\omega^\Gamma,$ and the last one is induced by the restrictions of $\gamma^\Gamma_E.$
\end{proof}

\section{Composition series of graphs}
\label{section_composition_series}

If $E$ is any graph and $K$ a field, a {\em (graded) composition series of length $n$} of $L_K(E)$ is a chain of (graded) ideals \[\{0\}=I_0 \lneq I_1\lneq \ldots\lneq I_n=L_K(E)\] such that the (graded) algebra $I_{i+1}/I_i$ is (graded) simple for all $i=0, \ldots, n-1.$ 
By Lemma \ref{lemma_graded_ideal_transitive}, requiring that $I_i$ is a (graded) ideal of $I_{i+1}$ for all $i=0,\ldots, n-1$ is equivalent to requiring that $I_i$ is a (graded) ideal of the entire algebra.  
The algebra $L_K(E)$ has a {\em (graded) composition series} if there is a positive integer $n$ such that $L_K(E)$ has a (graded) composition series of length $n$.  We also note that increasing, not necessarily finite, chains of graded ideals with simple quotients of specific type were considered in  \cite[Theorem 6.4]{Ranga_Fin_Gen}.

Theorem \ref{theorem_porcupine_quotient} enables us to characterize the existence of a graded composition series in purely graph-theoretic terms using the following definition.  

A graph $E$ has a {\em composition series of length $n$} if there is a chain of admissible pairs \[(\emptyset, \emptyset)=(H_0, S_0) \lneq (H_1, S_1)\lneq \ldots\lneq (H_n, S_n)=(E^0, \emptyset)\] such that the porcupine-quotient graph $(H_{i+1}, S_{i+1})/(H_i, S_i)$ is cofinal for all $i=0, \ldots, n-1.$ If $S_i=\emptyset$ for all $i,$ we write the above chain shorter as $\emptyset=H_0\lneq H_1\lneq \ldots\lneq H_n=E^0.$ The graph $E$ has a {\em composition series} if $E$ has a composition series of length $n$ for some positive integer $n.$  

For example, let $E$ be the graph from part (1)  of Example \ref{example_three_porcupine_quotients} and $H$ and $G$ be as in the same example. Then $\emptyset\leq H\leq G\leq E^0$ is a composition series of $E.$ If $E$ is the graph from part (2) of Example \ref{example_three_porcupine_quotients} and $H$ is as in that same example, then $(\emptyset, \emptyset)\leq (H, \emptyset)\leq (H, B_H)\leq (E^0, \emptyset)$ is a composition series of $E$.

Theorem \ref{theorem_porcupine_quotient}  has the following direct corollary. 

\begin{corollary}
If $E$ is any graph, the following conditions are equivalent. 
\begin{enumerate}[\upshape(1)]
\item The algebra $L_K(E)$ has a graded composition series. \hskip.2cm {\em (2)} The graph $E$ has a composition series.  
\end{enumerate} 
\label{corollary_composition_equivalent}
\end{corollary}

The existence of a composition series of a graph is equivalent to the existence of such series of both the porcupine and the corresponding quotient graph as we show next. 
We note that a similar claim has been shown 
for $\Gamma$-refinement monoids in \cite[Lemma 2.11]{Roozbeh_Alfilgen_Jocelyn}.

\begin{proposition} If $(H,S)$ is an admissible pair of a graph $E$, then $E$ has a composition series if and only if $P_{(H,S)}$ and $E/(H,S)$ have composition series.\label{proposition_porcupine_and_quotient}
\end{proposition}
\begin{proof}
By Corollary \ref{corollary_composition_equivalent}, it is sufficient to consider the graded ideals and graded composition series of the related Leavitt path algebras. Let $I=I(H,S).$

If $\{0\}=I_0\lneq \ldots\lneq I_n=L_K(E)$ is a graded composition series of $L_K(E),$ then it is direct to check that $\{0\}=I_0\cap I\lneq \ldots\lneq I_n\cap I=I$ produces a graded composition series of $I.$ Each term of this series is graded isomorphic to a graded ideal of $L_K(P_{(H,S)})$ and these graded ideals constitute a graded composition series of $L_K(P_{(H,S)})$. It is also direct to check that $\{I\}=(I_0+I)/I\lneq \ldots\lneq (I_n+I)/I=L_K(E)/I$ is a graded composition series of $L_K(E)/I.$ Each term of this series is graded isomorphic to a graded ideal of $L_K(E/(H,S))$ and the images of the terms of the series constitute a graded composition series of $L_K(E/(H,S)).$

Conversely, if $\{0\}=I'_0\lneq \ldots\lneq I'_n=L_K(P_{(H,S)})$ is a graded composition series of $L_K(P_{(H,S)}),$ the images $I_i$ of $I_i'$ for $i=0,\ldots,n$ under the graded isomorphism of $L_K(P_{(H,S)})$ and $I$ produce a graded composition series of $I.$  Similarly, if $\{0\}=J'_0\lneq \ldots\lneq J'_m=L_K(E/(H,S))$ is a graded composition series of $L_K(E/(H,S)),$ it uniquely determines the graded ideals $\{I\}=J_0/I\lneq \ldots\lneq J_m/I=L_K(E)/I$ of $L_K(E)/I$ which constitute a graded composition series of $L_K(E)/I.$ The ideals $I_0,\ldots, I_n=J_0, \ldots, J_m$ are graded ideals of $L_K(E)$ by Lemma  \ref{lemma_graded_ideal_transitive} and so \[\{0\}=I_0\lneq I_1\lneq\ldots\lneq I_n=I=J_0\lneq J_1\lneq\ldots\lneq J_m=L_K(E)\]
is a graded composition series of $L_K(E).$ 
\end{proof}

A {\em composition series of length $n$} of the graph monoid $M_E$ is a chain of order-ideals \[\{0\}=I_0 \lneq I_1\lneq \ldots\lneq I_n=M_E\] such that the monoid $I_{i+1}/I_i$ is simple (i.e., without any nontrivial and improper order-ideals) for all $i=0, \ldots, n-1.$  The monoid $M_E$ has a {\em composition series} if $M_E$ has a composition series of length $n$ for some 
positive integer $n.$  

We recall that $\Gamma$ is the infinite cyclic group on a generator $t$. A {\em composition series of length $n$} of the talented monoid $M_E^\Gamma$ is a chain of $\Gamma$-order-ideals \[\{0\}=I_0 \lneq I_1\lneq  \ldots\lneq I_n=M_E^\Gamma\] such that the $\Gamma$-monoid $I_{i+1}/I_i$ is simple  (i.e., without any nontrivial and improper $\Gamma$-order-ideals) for all $i=0, \ldots, n-1.$  The monoid $M_E^\Gamma$ has a {\em composition series} if $M_E^\Gamma$ has a composition series of length $n$ for some positive integer $n.$  

By \cite[Theorem 3.25]{Alfilgen_Jocelyn}, if $M_E^\Gamma$ has a composition series, then any two composition series have the same length (and the composition factors are isomorphic up to a permutation). This implies the second part of the following corollary. 

\begin{corollary}
If $E$ is any graph, the conditions from Corollary \ref{corollary_composition_equivalent} are equivalent to any of the conditions below. 
\begin{enumerate}[\upshape(1)]
\item[{\em (3)}] The monoid $M_E$ has a composition series. \hskip1cm
{\em (4)} The $\Gamma$-monoid $M_E^\Gamma$ has a composition series. 
\end{enumerate}
If these equivalent conditions hold, then each composition series of $E$, of $M_E$ and of $M_E^\Gamma$ and each graded composition series of $L_K(E)$ have the same length.  
\label{corollary_composition_equivalent_with_talented}
\end{corollary}
\begin{proof}
The first sentence follows directly from  Proposition \ref{proposition_monoids_of_porcupine_quotient}. Any two composition series of $E$ of lengths $m$ and $n$ respectively give rise to two composition series of $M_E^\Gamma$ by Proposition \ref{proposition_monoids_of_porcupine_quotient}. By \cite[Theorem 3.25]{Alfilgen_Jocelyn}, $m=n.$ Analogous arguments can be used for graded composition series of $L_K(E)$ and for composition series of $M_E.$ 
\end{proof}

By \cite[Theorem 3.29]{Alfilgen_Jocelyn}, if $M_E^\Gamma$ has a composition series, then there are no strictly increasing or strictly decreasing infinite chains of $\Gamma$-order-ideals. This result, Corollary \ref{corollary_composition_equivalent_with_talented}, and Proposition \ref{proposition_monoids_of_porcupine_quotient} have the following corollary. 

\begin{corollary}
If $E$ is any graph and if there is a sequence $(H_n, S_n), n=0,1,\ldots$ of admissible pairs of $E$ such that either \[(\emptyset, \emptyset)\lneq (H_0, S_0)\lneq (H_1, S_1)\lneq\ldots\;\;\mbox{ or }\;\;(E^0, \emptyset)\gneq (H_0, S_0)\gneq (H_1, S_1)\gneq\ldots\] holds and the  chain never becomes constant, then $E$ does not have a composition series.   
\label{corollary_equivalence}
\end{corollary}
\begin{proof}
Consider the $\Gamma$-order-ideals of the admissible pairs to obtain an infinite chain of either strictly increasing or strictly decreasing $\Gamma$-order-ideals of $M_E^\Gamma$ using Proposition \ref{proposition_monoids_of_porcupine_quotient}. By \cite[Theorem 3.29]{Alfilgen_Jocelyn}, $M_E^\Gamma$ does not have a composition series. By Corollary \ref{corollary_composition_equivalent_with_talented}, $E$ does not have a composition series.
\end{proof}

If $E$ is  a row-finite graph, the authors of \cite{Roozbeh_Alfilgen_Jocelyn} define a composition series of $M_E^\Gamma$ analogously as we do above (see 
\cite[Definition 2.8]{Roozbeh_Alfilgen_Jocelyn}) but relate it to admissible pairs of $E$ using the relative quotients (see Remark \ref{remark_relative_quotient}), not the porcupine-quotients. The next example illustrates  the differences between the two quotients on the $\Gamma$-monoid level if the order-units are considered. 
 
\begin{example}
\label{example_composition_talented}
Let $E, G,$ and $H$ be as in part (1) of Example \ref{example_three_porcupine_quotients}. Recall that $\emptyset\leq H\leq G\leq E^0$ is a graded composition series of $E$. The three related porcupine-quotients are below 
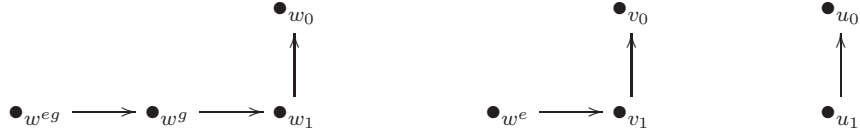
\[\xymatrix{   
&& \bullet_{w_0}    \\   
\bullet_{w^{eg}}\ar[r]&\bullet_{w^g}\ar[r] & \bullet_{w_1}\ar[u]  }\hskip2cm
\xymatrix{   
& \bullet_{v_0}    \\   
\bullet_{w^e}\ar[r] & \bullet_{v_1}\ar[u]  }\hskip2cm\xymatrix{   
\bullet_{u_0}  \\   
\bullet_{u_1}\ar[u]}\]
and their Leavitt path algebras are graded isomorphic to $\M_4(K)(0,1,2,3),$ $\M_3(K)(0,1,2),$ and $\M_2(K)(0,1)$
respectively. The usual matrix algebras are considered as graded
algebras here and the grading is given by: $x\in M_n(K)$ is in the $m$-th component of
$\M_n(K)(k_1, \ldots, k_n)$ if $x_{ij}\in K_{m-k_i+k_j}$ for all
$i,j=1,\ldots, n$ (more details can be found in \cite[Section
1.3]{Roozbeh_book} or  \cite[Section 2.1]{Lia_realization}).
In this example, the numbers in parenthesis following the usual matrix algebra notation correspond to the lengths of paths of the graphs ending at the sink of the graphs (see
\cite[Proposition 5.1]{Roozbeh_Lia_Ultramatricial}). 

The algebras $\M_4(K)(0,1,2,3),$ $\M_3(K)(0,1,2),$ and $\M_2(K)(0,1)$ are
graded isomorphic to the three quotients of graded ideals $I(H)/I(\emptyset), I(G)/I(H),$ and $L_K(E)/I(G)$ by Theorem
\ref{theorem_porcupine_quotient}. On the other hand, the three relative quotients are $
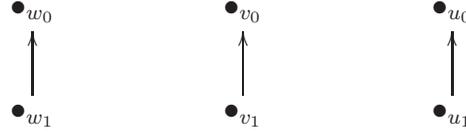
\xymatrix{   
\bullet_{w_0}  \\   
\bullet_{w_1}\ar[u]}\hskip2cm
\xymatrix{   
\bullet_{v_0}  \\   
\bullet_{v_1}\ar[u]}\hskip2cm\xymatrix{   
\bullet_{u_0}  \\   
\bullet_{u_1}\ar[u]}$
and the Leavitt path algebras of these  graphs are graded isomorphic to $\M_2(K)(0,1).$ Thus, the algebras of the first two relative quotients are not isomorphic the quotients $I(H)$ and $ I(G)/I(H)$ respectively.
 
The talented monoid of any of the six graphs above is isomorphic to $\Zset^+[t,t^{-1}]$ consisting of Laurent polynomials with nonnegative integer coefficients. However, if we consider the talented monoids together with their order-units (see \cite[Section 3.6.1]{Roozbeh_book}, \cite[Section 2.5]{Lia_realization}, or \cite[Section 1.1]{Roozbeh_Lia_Comparable} for relevant definitions), the triple 
\[(\Zset^+[t,t^{-1}], 1+t^{-1}+t^{-2}+t^{-3}),\;\; (\Zset^+[t,t^{-1}], 1+t^{-1}+t^{-2})\mbox{ and }(\Zset^+[t,t^{-1}], 1+t^{-1})\]
is different from the triple
$(\Zset^+[t,t^{-1}], 1+t^{-1}),\;\; (\Zset^+[t,t^{-1}], 1+t^{-1})\mbox{ and }(\Zset^+[t,t^{-1}], 1+t^{-1}).$
\end{example} 

\section{The four-color characterization of graded simple Leavitt path algebras}
\label{section_graded_simple}

We pause with the consideration of composition series until section \ref{section_comp_series_characterization}. In this section, we introduce the fourth type of vertices which are terminal in the same sense as the sinks and the vertices of cycles which are either without exists or extreme and show Theorem \ref{theorem_graded_simple}.

By \cite[Lemma 3.7.10]{LPA_book}, every vertex of a graph with finitely many vertices connects to a sink, a cycle with no exits, or an extreme cycle. However, in a graph with infinitely many vertices, that does not have to happen as it is the case for the graph below. \[\xymatrix{\bullet\ar@/_1pc/ [r] \ar@/^1pc/ [r] &\bullet\ar @/_1pc/ [r] \ar@/^1pc/ [r]& \bullet\ar @/_1pc/ [r] \ar@/^1pc/ [r]& \bullet\ar@{.}[r]&}\]

The following proposition generalizes \cite[Lemma 3.7.10]{LPA_book} to graphs of arbitrary cardinality.      
 
\begin{proposition}
If $E$ is any graph, each vertex of $E$ connects to a sink, an extreme cycle, a cycle without exits, or it is on an infinite path containing the vertices  $v_0>v_1>\ldots.$ 
\label{proposition_existence_of_the_four_types}
\end{proposition}
\begin{proof}
Let $v_0\in E^0$ be arbitrary. If $v_0$ is a sink or on a cycle which is extreme or without exits, the claim holds for $v_0.$ Otherwise, if $v_0$ is on a cycle $c$, then $c$ has an exit but it is not extreme. So, there is a path $p_0$ with $\so(p_0)=v_0$ and $\ra(p_0)\notin R(v_0).$ If $v_0$ is not on a cycle and as  $v_0$ is not a sink, $v_0$ emits edges and we let $p_0=e$ for any $e\in\so^{-1}(v_0).$ In either case, $v_1=\ra(p_0)\notin R(v_0),$ so $v_0>v_1.$   

Consider then $v_1.$ If $v_1$ is a sink, on a cycle without exits or on an extreme cycle, the claim holds for $v_1$ and, hence, for $v_0$ also. If not, then either $v_1$ is on a cycle emitting a path $p_1$ such that $\ra(p_1)\notin R(v_1),$ or $v_1$ is not on a cycle and it emits an edge in which case we let $p_1$ be that edge. In either case, $v_2=\ra(p_1)\in T(v_1)-R(v_1)$ which implies that $v_0>v_1>v_2.$ Continuing this process either terminates after finitely many steps resulting in a path from $v_0$ to a sink or a cycle which is either extreme or without exits, or the process does not terminate after finitely many steps and we obtain an infinite path containing vertices with the required properties.   
\end{proof} 

By Proposition \ref{proposition_existence_of_the_four_types}, the cofinality of a graph $E$ can be characterized in terms of the equivalence relation of $E^{\leq\infty}$ given by 
\begin{center}
$p \sim q\;\;$ if $\;R(p^0)=R(q^0).$ 
\end{center}
\begin{corollary}
A graph $E$ is cofinal if and only if the relation $\sim$ has only one equivalence class.
\label{corollary_cofinality_via_sim}
\end{corollary}
\begin{proof}
If $E$ is cofinal and $p,q\in E^{\leq\infty},$ then $p^0\subseteq R(q^0)$ by the cofinality of every vertex of $p^0,$ so $R(p^0)\subseteq R(q^0).$ Symmetrically, $R(q^0)\subseteq R(p^0).$ 

To show the converse, let $v\in E^0$ and $p\in E^{\leq\infty}.$ By Proposition \ref{proposition_existence_of_the_four_types}, there is an element $q$ of $E^{\leq\infty}$ such that $v\in q^0.$ Since $R(p^0)=R(q^0),$ $v\in R(p^0)$ which shows that $v$ is cofinal. 
\end{proof}

\subsection{Terminal paths}
The following definition leads us to the ``fourth primary color''. 

\begin{definition}
An infinite path $\alpha$ of a graph $E$ is {\em terminal} if no element of $T(\alpha^0)$ is an infinite emitter or on a cycle and if every infinite path $\beta$ with $\so(\beta)\in\alpha^0$ is such that $T(\beta^0)\subseteq R(\beta^0)$.
\label{definition_terminal_path} 
\end{definition}

If $\alpha$ is a terminal path, then $T(\alpha^0)\subseteq R(\alpha^0)$ holds. This implies that $T(\alpha^0)$ contains no sinks. 

Any infinite path in each of the two graphs below is terminal. Note that no vertex of the first graphs has a bifurcation. However, in the second graph, every vertex has a bifurcation.  
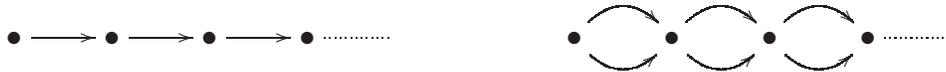
\[\xymatrix{\bullet\ar[r] &\bullet\ar[r] &
\bullet\ar[r] & \bullet\ar@{.}[r]&}\hskip2cm\xymatrix{\bullet\ar@/_1pc/ [r] \ar@/^1pc/ [r] &\bullet\ar @/_1pc/ [r] \ar@/^1pc/ [r]&
\bullet\ar @/_1pc/ [r] \ar@/^1pc/ [r]& \bullet\ar@{.}[r]&}\]

An infinite path containing infinitely many vertices does not have to be terminal. Indeed, no infinite path is terminal in any of the three graphs below. 
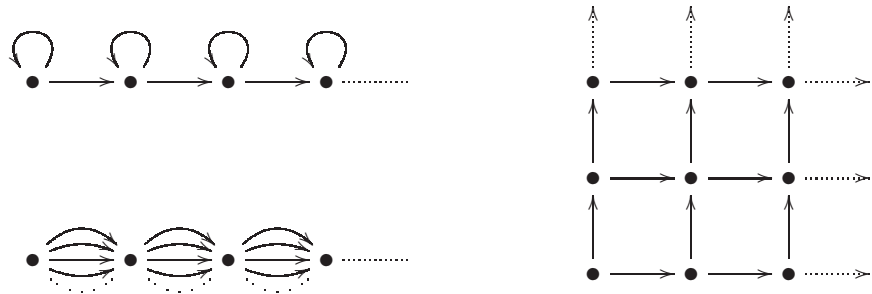
\[\xymatrix{\\ \bullet\ar@(ur,ul)\ar[r]& \bullet\ar@(ur,ul)\ar[r]&\bullet\ar@(ur,ul)\ar[r]&\bullet\ar@(ur,ul)\ar@{.}[r]&\\
\\
\bullet\ar@{.} @/_1pc/ [r] _{\mbox{ } } \ar@/_/ [r] \ar [r] \ar@/^/ [r] \ar@/^1pc/ [r] &\bullet\ar@{.} @/_1pc/ [r] _{\mbox{ } } \ar@/_/ [r] \ar [r] \ar@/^/ [r] \ar@/^1pc/ [r]&\bullet \ar@{.} @/_1pc/ [r] _{\mbox{ } } \ar@/_/ [r] \ar [r] \ar@/^/ [r] \ar@/^1pc/ [r]&\bullet\ar@{.}[r]&}
\hskip2cm
\xymatrix{&&&\\
\bullet\ar[r]\ar@{.>}[u] & \bullet\ar[r]\ar@{.>}[u]  & \bullet\ar@{.>}[r]\ar@{.>}[u] &\\
\bullet\ar[r] \ar[u]  & \bullet\ar[r] \ar[u] & \bullet \ar@{.>}[r] \ar[u] &\\
\bullet\ar[r] \ar[u]  & \bullet\ar[r] \ar[u] & \bullet \ar@{.>}[r] \ar[u] &}
\]

Lemma \ref{lemma_on_intersecting_terminal_paths} shows some properties of terminal paths.

\begin{lemma} Let $E$ be any graph and let 
$\alpha$ be a terminal path of $E.$ 
\begin{enumerate}[\upshape(1)]
\item Every infinite path $\beta$ originating at a vertex of $\alpha$ is terminal and $\alpha\sim\beta.$

\item If $\beta$ is an infinite path which contains a vertex $v$ of $\alpha$, then the suffix $\gamma$ of $\beta$ starting at $v$ is terminal and $\alpha\sim\beta\sim\gamma.$
\end{enumerate}
\label{lemma_on_intersecting_terminal_paths}
\end{lemma}
\begin{proof}
To show (1), let $\beta$ be an infinite path with $\so(\beta)\in \alpha^0.$ As $\alpha$ is terminal and $T(\beta^0)\subseteq T(\alpha^0),$ no element of $T(\beta^0)$ is an infinite emitter or on a cycle. If $\gamma$ is an infinite path originating at a vertex of $\beta$ and $p$ the part of $\beta$ from $\so(\beta)$ to $\so(\gamma),$ then $T(\gamma^0)\subseteq T((p\gamma)^0)\subseteq R((p\gamma)^0)=R(\gamma^0)$ where the second inclusion holds because $\alpha$ is terminal and the last equality holds since $\gamma^0\subseteq (p\gamma)^0$ and $(p\gamma)^0\subseteq R(\gamma^0)$. Hence, $\beta$ is terminal. If $v\in R(\alpha^0),$ let $u\in \alpha^0$ be such that $v\in R(u).$ As both $u$ and $\so(\beta)$ are on $\alpha,$ $u\geq \so(\beta)$ or $\so(\beta)\geq u.$ If $u\geq \so(\beta),$ then $u\in R(\beta^0),$ so $v\in R(\beta^0).$ If $\so(\beta)\geq u,$ then  $u\in T(\beta^0).$ As $\beta$ is terminal, $u\in T(\beta^0)\subseteq R(\beta^0).$ Thus, $v\in R(u)\subseteq R(\beta^0).$ This shows that $R(\alpha^0)\subseteq R(\beta^0).$ For the converse, let $v\in R(\beta^0)$ and let $u\in \beta^0$ be such that $v\in R(u).$ As $u\in T(\alpha^0)\subseteq R(\alpha^0),$ $v\in R(\alpha^0).$ This shows that $R(\alpha^0)=R(\beta^0)$ and so $\alpha\sim\beta.$

To show (2), assume that $v$ and $\gamma$ are as in the assumption of part (2). By part (1), $\gamma$ is terminal and $\alpha\sim\gamma$. As $\beta^0\subseteq R(\gamma^0)$ and $\gamma^0\subseteq \beta^0,$ $R(\beta^0)=R(\gamma^0).$ So, $\beta\sim\gamma.$ 
\end{proof}

\subsection{The four-color characterization of graded simple Leavitt path algebras}
\label{subsection_graded_simple}
Next, we formally introduce the notion of a ``cluster'' of vertices, mentioned in the introduction.  

\begin{definition}
A vertex $v$ of a graph $E$ is {\em terminal} if it is sink, on a cycle without exits, on an extreme cycle, or on a terminal path.  

Let $T_E$ be the set of terminal vertices. If $T_E\neq\emptyset,$ we define an equivalence relation on $T_E$ by 
\begin{center}
$v\approx w\;\;$ if $\;v\in p^0$ and $w\in q^0$ for $p,q\in E^{\leq\infty}$ such that $p\sim q.$     
\end{center}

The {\em cluster} of a terminal vertex $v$ is the equivalence class $\{w\in T_E\mid v\approx w\}.$  
\label{definition_clusters}
\end{definition}
It is direct to check that $\approx$ is reflexive, symmetric and transitive for vertices which are not on terminal paths. By Lemma \ref{lemma_on_intersecting_terminal_paths}, $\approx$ is transitive for vertices on terminal paths also. 

Let us consider some examples of clusters. 
For the first two graphs below, every vertex is terminal and each graph has only one cluster. The sink and the vertex on the cycle of the third graph are terminal and each is in its own one-element cluster.
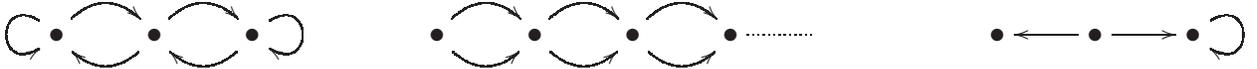
\[\xymatrix{{\bullet} \ar@/^1pc/ [r] \ar@(lu,ld) & {\bullet} \ar@/^1pc/[l]  \ar@/^1pc/ [r] & {\bullet}
\ar@/^1pc/ [l] \ar@(ur,dr) }\hskip2cm \xymatrix{\bullet\ar@/_1pc/ [r] \ar@/^1pc/ [r] &\bullet\ar @/_1pc/ [r] \ar@/^1pc/ [r]&
\bullet\ar @/_1pc/ [r] \ar@/^1pc/ [r]& \bullet\ar@{.}[r]&}\hskip2cm \xymatrix{{\bullet} & {\bullet} \ar[l]  \ar[r] & {\bullet} \ar@(ur,dr) }\] 

Lemma \ref{lemma_terminal_cluster} describes the cluster of any terminal vertex of a graph. By part (3) of Lemma \ref{lemma_terminal_cluster}, 
if the relation $\approx$ is considered only on the terminal vertices which are on extreme cycles, then it coincides with the relation from \cite[Definition 3.7.1]{LPA_book}.

\begin{lemma}
Let $v$ be a terminal vertex of a graph $E$ and let $C$ be its cluster.  One of the following four conditions holds.  
\begin{enumerate}[\upshape(1)]
\item The vertex $v$ is a sink. The element $v\in E^{\leq\infty}$ contains $v,$ it is a unique such element of $E^{\leq\infty}$ up to $\sim,$ and $C=\{v\}=T(v)=T(C).$  So, $\ol C=\ol{\{v\}}.$

\item The vertex $v$ is on a cycle $c$ without exits. The element $ccc\ldots\in E^{\leq\infty}$ contains $v,$ it is a unique such element of $E^{\leq\infty}$ up to $\sim$, and $C=c^0=T(c^0)=T(C).$ So, $\ol C=\ol{\{u\}}$ for any $u\in c^0.$

\item The vertex $v$ is on an extreme cycle $c.$ The element $ccc\ldots\in E^{\leq\infty}$ contains $v,$ it is a unique such element of $E^{\leq\infty}$ up to $\sim,$ and $C=T(c^0)=T(C).$ So, $\ol C=\ol{\{u\}}$ for any $u\in T(c^0).$  

\item The vertex $v$ is on a terminal path $\alpha.$ The element $\alpha\in E^{\leq\infty}$ contains $v,$ it is a  unique such element of $E^{\leq\infty}$ up to $\sim,$ $C=T(C)=\bigcup T(\beta^0)$ where the union is taken over terminal paths $\beta$ such that $\alpha\sim\beta,$ and $\ol C= \ol{\alpha^0}=\ol{\{u\}}$ for any $u\in C.$  
\end{enumerate}
\label{lemma_terminal_cluster} 
\end{lemma}
\begin{proof}
If $v$ is a sink and if $v\in p^0$ for some $p\in E^{\leq\infty},$ then $\ra(p)=v$ and $R(v)=R(p^0),$ so $v\sim p.$ Thus, if $w\approx v,$ then $v=w$, so $C=\{v\}.$ As 
$T(v)=\{v\},$ $T(C)=T(\{v\})=\{v\}=C$ and $\ol C=\ol{\{v\}}.$  
 
If $v$ is a vertex of a cycle $c$ without exits and if $v\in p^0$ for some $p\in E^{\leq\infty},$ then the only terminal vertices of $p$ are the vertices in $c^0$ and $R(c^0)=R(p).$ So, $ccc\ldots \sim p.$ If $w\approx v,$ then $w\in c^0,$ so $C=c^0.$ As $T(c^0)=c^0,$ $T(C)=T(c^0)=c^0=C$ and $\ol C=\ol{c^0}=\ol{\{u\}}$ for any $u\in c^0.$ 

If $v$ is a vertex of an extreme cycle $c,$ then all vertices in $T(v)$ are on extreme cycles which have the same root as $c.$ Hence, $T(c^0)\subseteq C.$ If $v\in p^0$ for some $p\in E^{\leq\infty},$ then the only terminal vertices of $p$ are the vertices on extreme cycles with the same root as $c$. Thus, $R(c^0)=R(p^0)$ which implies that $ccc\ldots\sim p.$ If $w\approx v,$ then $v$ and $w$ are on extreme cycles with the same roots. As the vertices of any such cycle are in $T(c^0),$ we have that $C\subseteq T(c^0).$ We already have the converse so $C=T(c^0).$ Thus, $T(C)=T(c^0)=C$ and $\ol C=\ol{c^0}=\ol{\{u\}}$ for any $u\in T(c^0).$ 
 
If $v$ is a vertex such that $v\in
\alpha^0$ for some terminal path $\alpha,$ then no vertex of $T(v)$ is on a cycle and it is neither a sink nor an infinite emitter. Hence, if $v\in p^0$ for some $p\in E^{\leq\infty},$ then the suffix $\beta$ of $p$ past $v$ is a terminal path such that  $\alpha\sim\beta\sim p$ by part (2) of Lemma \ref{lemma_on_intersecting_terminal_paths}. 

If $w\in T(\beta^0)$ for some terminal path $\beta$ such that $\beta\sim\alpha,$ then
$w\approx v,$ so $w\in C.$ Conversely, if 
$w\approx v,$ then $w$ is on some $q\in E^{\leq\infty}$ such that $q\sim\alpha.$ As $w$ is a terminal vertex, the suffix $\gamma$ of $q$ originating at $w$ is terminal and $\gamma\sim q\sim\alpha$ by part (2) of Lemma \ref{lemma_on_intersecting_terminal_paths}.
Hence, $w\in \bigcup T(\beta^0)$ where the union is taken over terminal paths $\beta$ such that $\beta\sim \alpha.$ If $U$ denotes this union, this shows that $C=U.$ As $U$ is hereditary, we have that $T(C)=T(U)=U=C.$

Next, we show that $\ol C=\ol{\alpha^0}.$
As $T(\alpha^0)\subseteq C,$ $\ol{\alpha^0}\subseteq \ol C.$ 
If $w\in \ol C$ is arbitrary, any infinite path originating at $w$ has a terminal suffix $\beta$ such that $\alpha\sim\beta$ by the previous paragraph and Lemma \ref{lemma_saturated_closure}. Since $\so(\alpha)\in R(\beta^0),$ $\beta$ contains a vertex in $T(\alpha^0).$ This shows that any infinite path in $\ol C$ contains a vertex of $T(\alpha^0).$ Since $w\in R(\beta^0)=R(\alpha^0)\subseteq R(T(\alpha^0)),$ we have that $\ol C\subseteq R(T(\alpha^0)).$ So, 
$T(\alpha^0)\subseteq C\subseteq \ol C\subseteq R(T(\alpha^0)).$ 
As $\ol C$ contains no infinite emitters, we can use Lemma \ref{lemma_saturated_closure} to conclude that $\ol C=\ol{\alpha^0}.$ 
If $u\in C,$ then $u$ is the source of a terminal path $\gamma$ such that $\gamma\sim \alpha$ and the same argument applies to $\gamma$ instead of $\alpha$ to show that $\ol C=\ol{\gamma^0}.$ The relation $T(\gamma^0)\subseteq T(u)$ implies that $\ol C=\ol{\gamma^0}=\ol{T(\gamma^0)}\subseteq \ol{T(u)}=\ol{\{u\}}.$   
\end{proof}

By \cite[Theorem 5.7]{Tomforde}, $L_K(E)$ is graded simple if and only if $E$ is cofinal. In Theorem \ref{theorem_graded_simple}, we characterize graded simplicity of $L_K(E)$ with the properties of $E$ presented in terms of the four primary colors.

\begin{theorem}
Let $E$ be a graph and $K$ be a field. The following conditions are equivalent. 
\begin{enumerate}[\upshape(1)]
\item $L_K(E)$ is graded simple (equivalently, $E$ is cofinal).  

\item The set of terminal vertices is nonempty and it consists of a single cluster $C$ such that $E^0$ is the (hereditary and) saturated closure of $C.$ 

\item Exactly one of the following holds. 
\begin{enumerate}[\upshape(a)]
\item The set $E^0$ is the (hereditary and) saturated closure of a sink. In this case, $E$ is row-finite and acyclic and $E^0=R(v)$ for a sink $v$.    

\item The set $E^0$ is the (hereditary and) saturated closure of $c^0$ for a cycle $c$ without exits. In this case, $E$ is row-finite, $E^0=R(c^0),$ and $c$ is the only cycle in $E$.    

\item  The set $E^0$ is the hereditary and saturated closure of $c^0$ for an extreme cycle $c.$ In this case, every cycle of $E$ is extreme, every infinite emitter is on a cycle, and $E^0=R(c^0)$.

\item The set $E^0$ is the hereditary and saturated closure of $\alpha^0$ for a terminal path $\alpha.$ In this case, $E$ is acyclic and row-finite and $E^0=R(\alpha^0)$.
\end{enumerate}  
\end{enumerate}
\label{theorem_graded_simple}
\end{theorem}
\begin{proof}
To show (1) $\Rightarrow$ (2), assume that $E$ is cofinal. If $v$ is an infinite emitter and $\ra(\so^{-1}(v))$ is not contained in $R(v)$, then the saturated closure of the hereditary set $T(\ra(\so^{-1}(v)))-R(v)$ is a proper and nontrivial hereditary and saturated set, so this cannot happen. Hence, $\ra(\so^{-1}(v))\subseteq R(v)$ so every infinite emitter is on a cycle. Similarly, if there is a cycle $c$ emitting a path $p$ such that $\ra(p)\notin R(c^0),$ then the saturated closure of $T(\ra(p))$ is a proper and nontrivial hereditary and saturated set. Hence, every cycle of $E$ is either extreme or without exits. As an infinite emitter cannot be on a cycle without exits, every infinite emitter is on an extreme cycle. 

Next, we claim that the set $T_E$ of terminal vertices is nonempty. This is clear if a vertex of $E$ connects to a sink, an extreme or cycle without exits. Otherwise, by Proposition \ref{proposition_existence_of_the_four_types}, every vertex of $E$ is on an infinite path containing infinitely many vertices.
As every cycle is extreme or without exits and every infinite emitter is on a cycle, this condition implies that $E$ is a row-finite and acyclic graph. Thus, if $\alpha$ is an infinite path, $T(\alpha^0)$ contains neither vertices on cycles nor infinite emitters. Hence, to show that $\alpha$ is terminal, it remains to show that $T(\beta^0)\subseteq R(\beta^0)$ for any infinite path $\beta$ with $\so(\beta)\in \alpha^0$. Assume, on the contrary, that there is $v\in T(\beta^0)$ such that $v\notin R(\beta^0)$ for one such $\beta.$ In that case, the saturated closure of $T(v)$ is nontrivial and proper ($\so(\beta)\notin \ol{\{v\}}$ by Lemma \ref{lemma_saturated_closure}). This is a contradiction, so $\alpha$ is terminal. As $\alpha^0\subseteq T_E,$ $T_E$ is nonempty.   

If $v\in T_E,$ then the cluster $C$ of $v$ is the only cluster in $E$ by the  cofinality of $E$. By Lemma \ref{lemma_terminal_cluster}, $T(C)=C,$ so the saturated closure $\ol C$ of $C$ is a nonempty hereditary and saturated set in $E^0.$ By the cofinality of $E,$ $E^0=\ol C.$ 

The implication (2) $\Rightarrow$ (3)
follows directly from Lemma \ref{lemma_terminal_cluster}. In the case that $C$ consists of vertices on terminal paths, $\ol C=\ol{\alpha^0}$ for a terminal path $\alpha$ by Lemma \ref{lemma_terminal_cluster}, so $E^0=\ol C$ implies that $E^0=\ol{\alpha^0}.$ As $T(\alpha^0)\subseteq R(\alpha^0),$  $R(T(\alpha^0))\subseteq R(R(\alpha^0))=R(\alpha^0).$ The converse $R(\alpha^0)\subseteq R(T(\alpha^0))$ trivially holds and so $E^0=\ol{\alpha^0}\subseteq R(T(\alpha^0))=R(\alpha^0).$ Thus, $E^0=R(\alpha^0).$ 

To show that (3) $\Rightarrow$ (1), we assume that (3) is true and show that the relation $\sim$ has only one equivalence class. By Corollary \ref{corollary_cofinality_via_sim}, this implies that $E$ is cofinal.

If (3a) holds, the relation $E^0=\ol{\{v\}}$ and Lemma \ref{lemma_saturated_closure} imply that there are neither other sinks, infinite emitters, cycles, nor infinite paths. Thus, every element of $E^{\leq \infty}$ is a finite path ending at $v.$ For any such path $p,$ $R(p^0)=R(v),$ so $p\sim v.$

If (3b) holds, the relation $E^0=\ol{c^0}$
and Lemma \ref{lemma_saturated_closure} imply that there are neither sinks, infinite emitters, nor cycles other than $c,$ and that any $p\in E^{\leq \infty}$ consists of a finite path reaching a vertex $v$ of $c$ followed by $ccc\ldots$ if $c$ is considered to start at $v$. As $R(p^0)=R(c^0)$ for any such path $p,$ $p\sim ccc\ldots.$  

If (3c) holds, the relation $E^0=\ol{c^0}$ and Lemma \ref{lemma_saturated_closure} imply that  there are no sinks, that every infinite emitter is in $T(c^0),$ and that every cycle is extreme with vertices in $T(c^0).$ Thus, every $p\in E^{\leq \infty}$ is a finite path followed by an infinite suffix with vertices in $T(c^0)$ or a finite path ending in an infinite emitter in $T(c^0).$ As $R(p^0)=R(c^0)$ for any such path $p,$  $p\sim ccc\ldots.$  

If (3d) holds, the relation $E^0=\ol{\alpha^0}$ and Lemma \ref{lemma_saturated_closure} imply
that there are neither sinks, infinite emitters, nor cycles and that any $p\in E^{\leq\infty}$ contains a vertex of $T(\alpha^0).$ Let $q$ be a path from a vertex of $\alpha^0$ to a vertex of $p$ and let $\beta$ be the suffix of $p$ originating at $\ra(q).$ By Lemma \ref{lemma_on_intersecting_terminal_paths}, $q\beta$ is terminal and $\alpha\sim q\beta.$ Thus, $p\sim\beta\sim q\beta\sim\alpha.$  
\end{proof}

The corollary below follows from Theorem \ref{theorem_graded_simple} and the porcupine-quotient construction. We use this corollary in the proofs of Theorem \ref{theorem_two_types} and Corollary \ref{corollary_one_type}.

\begin{corollary}
Let $E$ be any graph. 
\begin{enumerate}[\upshape(1)]
\item If $v$ is a sink or an infinite emitter not on a cycle, then there are admissible pairs $(H,S)$ and $(G,T)$ of $E$ such that $(G,T)/(H,S)$ is cofinal and that $v$ is a sink of $(G,T)/(H,S).$

\item If $c$ is a cycle of $E,$ then there are admissible pairs $(H,S)$ and $(G,T)$ of $E$ such that $(G,T)/(H,S)$ is cofinal and that $c$ is a cycle of $(G,T)/(H,S)$ which is extreme in $(G,T)/(H,S)$ if $c$ contains a vertex of another cycle of $E$ and which is without exits in $(G,T)/(H,S)$ otherwise. 
    
\item If $\alpha$ is an infinite path 
such that $T(\alpha^0)$ contains neither sinks, infinite emitters, nor vertices on cycles, then there are admissible pairs $(H,S)$ and $(G,T)$ of $E$ such that $(G,T)/(H,S)$ is cofinal and that 
$\alpha$ is a terminal path of $(G,T)/(H,S).$
\end{enumerate}
\label{corollary_construction_of_terminal_elements} 
\end{corollary}
\begin{proof}
To show (1), let $G=\ol{\{v\}},$ $H=\ol{\ra(\so^{-1}(v))}$ (possibly empty), and $T=S=\emptyset.$ If $v$ is a sink, then $G$ does not contain any infinite emitters, and if $v$ is an infinite emitter not in a cycle, then $v$ is the only infinite emitter in $G-H.$ In either case, $G\cap B_H^G=\emptyset.$ Thus, $(G,T)/(H,S)$ contains no vertices of the form $v',$ so $v$ is the only sink of $(G,T)/(H,S).$ The vertices of $G-H$ are in $R(v).$ If $p$ is a path such that $w^p$ is a vertex of $(G,T)/(H,S)$, then $w^p$ is in the root $R^{(G,T)/(H,S)}(v)$ of $v$ in $(G,T)/(H,S).$ Hence, 
$((G,T)/(H,S))^0=R^{(G,T)/(H,S)}(v).$
The graph $(G,T)/(H,S)$ is row-finite, acyclic, and without infinite paths. By Lemma \ref{lemma_saturated_closure}, $((G,T)/(H,S))^0=\ol{T^{(G,T)/(H,S)}(v)},$ so $(G,T)/(H,S)$ is cofinal by Theorem \ref{theorem_graded_simple}. 

To show (2), let $G=\ol{c^0}.$ The set $T(c^0)-R(c^0),$ possibly empty, is hereditary, so its saturated closure $H$ is hereditary and saturated. Let $T=\emptyset$ and $S=G\cap B_H^G.$ By the definition of $S,$ no vertices of the form $v'$ are in $(G,T)/(H,S).$ Similarly as in part (1), $((G,T)/(H,S))^0=R^{(G,T)/(H,S)}(c^0).$ The set $((G,T)/(H,S))^0-T^{(G,T)/(H,S)}(c^0)$ contains no infinite emitters and every infinite path with vertices in this set eventually reaches a vertex of $T^{(G,T)/(H,S)}(c^0)$ by the definition of $G$ and $H.$ By Lemma \ref{lemma_saturated_closure}, $((G,T)/(H,S))^0=\ol{T^{(G,T)/(H,S)}(c^0)},$ so $(G,T)/(H,S)$ is cofinal by Theorem \ref{theorem_graded_simple}.
If $c$ contains a vertex of another cycle of $E$, then $c$ has exits in $(G,T)/(H,S)$ and, as 
$((G,T)/(H,S))^0=R^{(G,T)/(H,S)}(c^0),$ $c$ is extreme in $(G,T)/(H,S).$ If $c$ contains no vertex of another cycle of $E$, then $c$ is without exits in $(G,T)/(H,S)$ by the definition of $G$ and $H.$ 

To show (3), let $G=\ol{\alpha^0}.$ Let  $V=\bigcup \left(T(\beta^0)-R(\beta^0)\right)$ where the union is taken over infinite paths $\beta$ originating in a vertex of $\alpha$ (possibly empty), let $H=\ol{V},$ and $T=S=\emptyset.$ Since $G$ contains no infinite emitters, $G\cap B_H^G=\emptyset,$ so $(G,T)/(H,S)$ contains no vertices of the form $v'.$  
By the definition of $G$ and $H$, $\alpha$ is terminal in $(G,T)/(H,S)$ and $((G,T)/(H,S))^0=R^{(G,T)/(H,S)}(T^{(G,T)/(H,S)}(\alpha^0)).$ The graph $(G,T)/(H,S)$ is row-finite, with neither sinks nor cycles, and any infinite path contains a vertex of $T^{(G,T)/(H,S)}(\alpha^0),$ so 
$((G,T)/(H,S))^0$ is the saturated closure of $T^{(G,T)/(H,S)}(\alpha^0)$ by Lemma \ref{lemma_saturated_closure}. By Theorem \ref{theorem_graded_simple}, 
$((G,T)/(H,S))$ is cofinal.  
\end{proof}

As a side result, we note that Theorem \ref{theorem_graded_simple} implies that  purely infinite simplicity and its graded version are equivalent for Leavitt path algebras. We review some definitions related to these concepts. An idempotent $u$ of a ring $R$ is {\em finite} if $uR$ is not isomorphic to a proper direct summand of itself. A simple ring $R$ is {\em purely infinite simple} if every nontrivial one-sided ideal contains an infinite idempotent. In the graded case, a homogeneous idempotent $u$ of a graded ring $R$ is {\em finite} if $uR$ is not graded isomorphic to a proper graded direct summand of itself. A graded simple ring $R$ is {\em graded purely infinite simple} if every nontrivial one-sided graded ideal contains an infinite homogeneous idempotent (see \cite[Proposition 3.8.8]{LPA_book} for equivalent conditions to being purely infinite simple).  

\begin{corollary}
Let $E$ be a graph and let $K$ be a field. The following conditions are equivalent. 
\begin{enumerate}[\upshape(1)]
\item The algebra $L_K(E)$ is graded purely infinite simple. 

\item The set $E^0$ is the hereditary and saturated closure of $c^0$ for an extreme cycle $c$ (i.e., $E$ satisfies condition (3c) of Theorem \ref{theorem_graded_simple}).  

\item The graph $E$ is cofinal, every cycle of $E$ has an exit, and every vertex of $E$ connects to a cycle.

\item The algebra $L_K(E)$ is purely infinite simple. 
\end{enumerate}      
\label{corollary_graded_pis}
\end{corollary}
\begin{proof}
To show (1) $\Rightarrow$ (2), assume that  $L_K(E)$ is graded purely infinite simple. Then $L_K(E)$ is graded simple, so one part of condition (3) of Theorem \ref{theorem_graded_simple} holds. If conditions (3a), (3b) or (3d) hold, then $L_K(E)$ is directly finite by \cite[Theorem 4.12]{Lia_traces}, so no idempotent is infinite. This shows that condition (3c) necessarily has to hold. 

To show the converse (2) $\Rightarrow$ (1), assume that (2) holds for $E.$ By the graded version of \cite[Proposition 3.1.7]{LPA_book}, it is sufficient to show that for every homogeneous and nonzero $a\in L_K(E),$ there are homogeneous $x,y\in L_K(E)$ such that $xay$ is an infinite idempotent. As every vertex connects to an extreme cycle, every vertex is an infinite idempotent by \cite[Proposition 3.1.6]{LPA_book}. In addition, for a homogeneous element $a\neq 0,$ there are paths $p$ and $q$ and $0\neq k\in K$ such that $paq=kv$ for some $v\in E^0$ by \cite[Theorem 2.2.11]{LPA_book}. Thus, we can take $x=k^{-1}p$ and $y=q.$

The implication (2) $\Rightarrow$ (3) follows from Theorem \ref{theorem_graded_simple}. Conversely, if (3) holds, then $E$ is cofinal, so exactly one condition from part (3) of Theorem \ref{theorem_graded_simple} holds. Since every cycle of $E$ has an exit, it is not condition (3b). As every vertex of $E$ connects to a cycle, it is neither (3a) nor (3d). Hence, it is (3c) and so condition (2) of the corollary holds. 

The equivalence of (3) and (4) is shown in \cite[Theorem 11]{Gene_Gonzalo_pis}. 
\end{proof}

\section{Constructive characterization of a composition series}
\label{section_comp_series_characterization}
 
Let $\Sink$ denote the hereditary and saturated closure of the set of sinks, $\NE$ denote the hereditary and saturated closure of the set of vertices on cycles without exits, $\EC$ denote the hereditary and saturated closure of the set of vertices on extreme cycles, and $I(\Ter_{fin})$ be the ideal generated by the union $\Sink\cup\NE\cup\EC$. For graphs with finitely many vertices, $I(\Ter_{fin})=I_{lce}$ defined as in \cite[Definition 3.7.8]{LPA_book}. The 
Leavitt path algebra of the graph
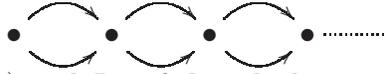
\[\xymatrix{\bullet\ar@/_1pc/ [r] \ar@/^1pc/ [r] &\bullet\ar @/_1pc/ [r] \ar@/^1pc/ [r]&
\bullet\ar @/_1pc/ [r] \ar@/^1pc/ [r]& \bullet\ar@{.}[r]&}\]
is graded simple and both $I(\Ter_{fin})$ and $I_{lce}$ of this algebra are trivial.  

Let $\Ter_\infty$ denote the hereditary and saturated closure of the set of vertices on terminal paths. Let $\Ter(E)$ denote the hereditary and saturated closure of $\Sink \cup \NE\cup \EC \cup \Ter_\infty$ (equivalently, the saturated closure of the hereditary set $T_E$ of terminal vertices) of a graph $E.$ If $E^0$ is finite, $I(\Ter(E))$ is $I_{lce}$. 

\begin{proposition}
For any graph $E,$ let $\mathcal C$ be the set of the clusters of $E.$ For any $C\in \mathcal C,$ the ideal $I(C)$ generated by $C$ is a graded simple algebra and  
\[I(\Ter(E))= I(\Sink)\oplus I(\NE)\oplus I(\EC)\oplus I(\Ter_\infty)=\bigoplus_{C\in\mathcal  C}I(C).\]
\label{proposition_ter_ideal} 
\end{proposition}

\begin{proof}
For $C\in \mathcal C,$ condition (2) of Theorem \ref{theorem_graded_simple} holds for the porcupine graph $P_{(\ol C, \emptyset)}$ of $\ol C.$ Hence, $P_{(\ol C, \emptyset)}$  is cofinal, so $I(C)$ is graded simple. The sets $T_E\cap \Sink, T_E\cap\NE ,$ $T_E\cap\EC,$ and $T_E\cap\Ter_\infty$ are mutually disjoint and different clusters are also mutually disjoint. So, the proposition follows from \cite[Proposition 2.4.7]{LPA_book} stating that if $V_i\subseteq E^0$ for $i\in I$ are pairwise disjoint and if $H_i=\ol{V_i}$ for $i\in I,$ then $I(\bigcup_{i\in I}V_i)=I(\bigcup_{i\in I}H_i)=\bigoplus_{i\in I}I(H_i)=\bigoplus_{i\in I}I(V_i).$ 
\end{proof}

\begin{remark}
In \cite{Cam_et_al}, the authors consider the set $P_{b^\infty}$ as the set of $v\in E^0$ such that $T(v)$ contains infinitely many vertices with bifurcations or an infinite emitter and generalize $I_{lce}$ by considering its extension by the ideal generated with the set $P_{b^\infty}$. This generalization is successful in the sense that every vertex connects to an element of $\Sink\cup\NE\cup\EC\cup  P_{b^\infty}$ (see \cite[Lemma 2.3]{Cam_et_al}). However, the elements of the set $P_{b^\infty}$ may not be terminal in the sense we are interested in. Also, the set $P_{b^\infty}$ may not be disjoint from $\EC$ because of the infinite emitters on extreme cycles, so we do not have a direct sum decomposition as in \cite[Theorem 3.7.9]{LPA_book} or in the proposition above. 
\end{remark}

Lemma below exhibits a group of necessary conditions for a graph to have a composition series. 

\begin{lemma} If a graph $E$ has a composition series, then the following holds. 
\begin{enumerate}[\upshape(a)]
\item $\Ter(E)$ is nonempty. 

\item The set of terminal vertices of $E$ contains finitely many clusters.

\item The set of breaking vertices of $\Ter(E)$ is finite.
\end{enumerate} 
\label{lemma_for_comp_series} 
\end{lemma}
\begin{proof}
If $E^0$ is finite, $\Ter(E)$ is nonempty since there is either a sink, a cycle without exits, or an extreme cycle by \cite[Lemma 3.7.10]{LPA_book}. If $E^0$ is infinite and there are neither sinks, extreme cycles, nor cycles without exits, then every vertex is on an infinite path containing an infinite and strictly decreasing chain of vertices by Proposition \ref{proposition_existence_of_the_four_types}. For brevity, let us say that such an infinite path is {\em strictly decreasing}. We claim that there is a strictly decreasing infinite path which is terminal. Assume, on the contrary, that no strictly decreasing infinite path is terminal. We consider the following cases: 
$T(\alpha^0)\nsubseteq R(\alpha^0)$ for all strictly decreasing infinite paths $\alpha$ and $T(\alpha^0)\subseteq R(\alpha^0)$ for some strictly decreasing infinite path $\alpha.$ 

In the first case, let $\alpha_0$ be a strictly decreasing infinite path and let  $H_0=\ol{\alpha_0^0}$. As $T(\alpha_0^0)\nsubseteq R(\alpha_0^0),$ there is a vertex $v_0\in T(\alpha_0^0)-R(\alpha_0^0).$ As $v_0$ does not connect to a sink or a cycle which is extreme or without exits, $v_0$ is the source of a strictly decreasing infinite path $\alpha_1.$ Since $v_0\notin R(\alpha_0^0),$ no vertex of $\alpha_1$ is in $R(\alpha_0^0).$ Hence, no vertex of $\alpha_0$ is in $T(\alpha_1^0).$ Thus, if we let $H_1=\ol{\alpha_1^0},$ we have that $\so(\alpha_0)\notin H_1$ by Lemma \ref{lemma_saturated_closure}. Since $\alpha_1^0\subseteq T(\alpha_0^0),$ we have that $H_1\subseteq H_0.$ So, $H_1\subsetneq H_0.$

As $T(\alpha_1^0)\nsubseteq R(\alpha_1^0),$ there is a vertex $v_1\in T(\alpha_1^0)- R(\alpha_1^0).$ Having $v_1,$ we can obtain $\alpha_2$ in the same way we obtained $\alpha_1$ having $v_0.$ Thus, for $H_2=\ol{\alpha_2^0},$ we have that 
$H_2\subseteq H_1$ and $\so(\alpha_1)\in H_1-H_2.$ By our assumptions, this process does not terminate, so we obtain a chain $H_0\supsetneq H_1\supsetneq H_2\supsetneq\ldots$. Thus, $E$ has no composition series by Corollary \ref{corollary_equivalence} and we reach a contradiction. 
   
In the second case, let $\alpha_0$ be a strictly decreasing infinite path such that $T(\alpha_0^0)\subseteq R(\alpha_0^0).$ Since $\alpha_0$ is not terminal, there is $v_0\in \alpha_0^0$ such that one of the three conditions holds: (1) $v_0$ emits a path whose range is in a cycle $c$, (2) $v_0$ emits a path whose range is an infinite emitter $v$ which is not on a cycle, or (3) $v_0$ does not connect to infinite emitters or vertices on cycles and it emits an infinite path $\beta$ such that $T(\beta^0)\nsubseteq R(\beta^0).$ In each case, we aim to find $w_0\in \alpha_0^0$ such that $v_0\notin H_0=\ol{\{w_0\}}.$

If (1) holds, $c^0\subseteq T(\alpha_0^0)$  implies that $T(c^0)\subseteq T(\alpha_0^0)\subseteq R(\alpha_0^0).$
The cycle $c$ is not extreme nor without exits, so there is a path $p$ with $\so(p)\in c^0$ and $\ra(p)\notin R(c^0).$ 
As $\ra(p)\in R(\alpha^0_0),$ such path $p$ can be chosen so that $w_0=\ra(p)\in \alpha_0^0-R(c^0).$ The condition $w_0\notin R(c^0)$ implies that $c^0\nsubseteq T(w_0).$
The vertices $v_0$ and $w_0$ are both on $\alpha_0,$ so either $v_0\leq w_0$ or $v_0>w_0.$ Since $v_0\in R(c^0)$ and $c^0\nsubseteq T(w^0),$ $v_0>w_0.$ As $c^0\nsubseteq T(w^0),$ $c^0\nsubseteq H_0$ by Lemma \ref{lemma_saturated_closure}. So,  $c^0\subseteq T(v_0)$ implies that $v_0\notin H_0.$ 

If (2) holds, $v\in T(\alpha_0^0)\subseteq R(\alpha_0^0),$ so $v$ emits a path $p$ with $w_0=\ra(p)\in \alpha^0_0.$ The vertices $v_0$ and $w_0$ are both on $\alpha_0,$ so either $v_0\leq w_0$ or $v_0>w_0.$ Since $v$ is not on a cycle, $v_0>w_0$ and $v\notin T(w_0).$ The relation $v\notin T(w_0)$ implies that $v\notin H_0$ by Lemma \ref{lemma_saturated_closure}. As $v\in T(v_0),$ we have that $v_0\notin H_0.$

If (3) holds, there is $v\in \beta^0$ which emits a path with the range in $T(\beta^0)-R(\beta^0).$ As $T(\beta^0)\subseteq  T(\alpha_0^0)\subseteq R(\alpha_0^0),$  there is a path $p$ with $\so(p)=v,$ $w_0=\ra(p)\in \alpha_0^0-R(\beta^0).$ The vertices $v_0$ and $w_0$ are both on $\alpha_0,$ so either $v_0\leq w_0$ or $v_0>w_0.$ Since $w_0\notin R(\beta^0)$ and $v_0\in R(\beta^0),$  $v_0>w_0.$ The condition $w_0\notin R(\beta^0)$ implies that no vertex of $\beta$ is in $T(w_0).$ By Lemma \ref{lemma_saturated_closure}, $v_0\notin H_0.$  

Let $\alpha_1$ be the suffix of $\alpha$ originating at $w_0.$ As no strictly decreasing infinite path is terminal, we have that every suffix of $\alpha_0$ is not terminal. So, $\alpha_1$ is not terminal. In addition, $T(\alpha_1^0)\subseteq T(\alpha_0^0)\subseteq R(\alpha^0_0)\subseteq R(\alpha_1^0),$ so we can repeat the construction and let $v_1$ be a vertex of $\alpha_1$ with the same properties as $v_0$ for $\alpha_0,$ let $w_1$ be obtained analogously to $w_0$ so that $v_1\in H_0$ is not in $H_1=\ol{\{w_1\}}.$ As $w_1\in T(w_0),$ $T(w_1)\subseteq T(w_0)$ which implies that $H_1\subseteq H_0.$ Hence, $H_1\subsetneq H_0.$ Continuing in this manner, we obtain a chain $H_0\supsetneq H_1\supsetneq \ldots $ 
which does not terminate because $\alpha_n$ is not terminal for each $n$. By Corollary \ref{corollary_equivalence}, $E$ has no composition series. So, we reach a contradiction. 

As we reach a contradiction in both cases, there is a strictly decreasing infinite path  $\alpha$ which is terminal. So, $\alpha^0\subseteq \Ter(E)$ implying that $\Ter(E)\neq\emptyset.$ 
 
If (b) fails, index the clusters by an infinite cardinal $\lambda$ and let $H_n$ be the hereditary and saturated closure of the vertices in the first $n$ clusters. The chain $H_0\subsetneq H_1\subsetneq\ldots$ does not terminate since $\lambda$ is infinite. By Corollary \ref{corollary_equivalence} and this fact, $E$ has no composition series.  

To show part (c), note that $E/(\Ter(E),\emptyset)$ has a composition series by 
Proposition \ref{proposition_porcupine_and_quotient}.
As $B_{\Ter(E)}$ corresponds to a set of sinks in $E/(\Ter(E),\emptyset)$ and the number of sinks of $E/(\Ter(E),\emptyset)$ 
is finite by part (b), $B_{\Ter(E)}$ is finite.      
\end{proof} 

Using Lemma \ref{lemma_for_comp_series}, it is not difficult to construct graphs which do not have composition series. For example,
each of the following three graphs fails exactly one of the three conditions of  Lemma \ref{lemma_for_comp_series}. The symbol $\infty$ in the last graph indicates that a vertex emitting the edge labeled by this symbol emits infinitely many edges to the sink of the graph.   
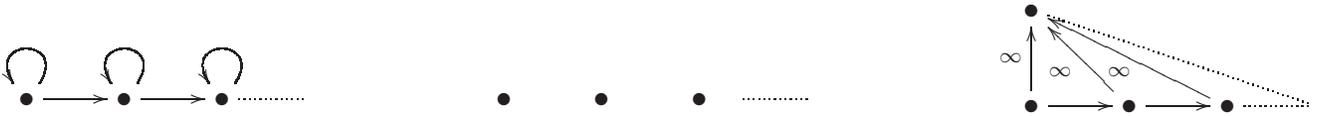
\[\xymatrix{\\\bullet\ar@(ur,ul)\ar[r]&\bullet\ar@(ur,ul)\ar[r]&\bullet\ar@(ur,ul)\ar@{.}[r]&}\hskip2.2cm\xymatrix{\\\bullet&\bullet&\bullet\;\;\;\ar@{.}[r]&}\hskip2.2cm \xymatrix{\bullet &&\\
\bullet\ar[r]\ar[u]^{\infty} &\bullet \ar[r] 
\ar[ul]^{\infty} & \bullet\ar@{.}[r] \ar[ull]^{\infty} & \ar@{.}[ulll]
}\]

In addition, the graph below satisfies all three conditions of Lemma \ref{lemma_for_comp_series} ($v_0$ is a terminal vertex, the cluster $\{v_0\}$ is the only cluster, and the graph is row-finite, so part (c) trivially holds). 
However, this graph does not have a composition series because $\emptyset\leq \{v_0\}\leq \{v_0, v_1\}\leq\ldots$ is an increasing chain such that the porcupine-quotient graph of any two consecutive terms is cofinal.  

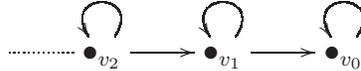
\[\xymatrix{\ar@{.}[r]&\bullet_{v_2}\ar@(ur,ul)\ar[r]& \bullet_{v_1}\ar@(ur,ul)\ar[r]&\bullet_{v_0}\ar@(ur,ul)}\]

The main result of this section, Theorem  \ref{theorem_comp_series}, shows that the four graphs above have a complete
list of features which obstruct the existence of a composition series of a graph. This result also provides a way of constructing a composition series if it exits. 

\begin{definition}
For a graph $E$, we define the {\em composition quotients} $F_n$ of $E$ as follows. 

Let $F_0=E.$ If $\Ter(F_n)\subsetneq F_n^0,$ we let 
\[F_{n+1}=F_n/(\Ter(F_n), B_{\Ter(F_n)}).\]
If $\Ter(F_n)=F_n^0,$ we let $F_{n+1}=F_{n+2}=\ldots =\emptyset.$
\end{definition}
Note that the case $\Ter(F_n)=\emptyset$ for some $n$ implies that $F_m=F_n$ for every $m\geq n.$

\begin{theorem}
The following conditions are equivalent for a graph $E$. 
\begin{enumerate}[\upshape(1)]
\item The graph $E$ has a composition series. 
\item The following holds.  
\begin{enumerate}[\upshape(i)]
\item Conditions (a), (b), and (c) of Lemma \ref{lemma_for_comp_series} hold for the composition quotient $F_n$ for each $n$ for which $F_n\neq\emptyset.$

\item There is a nonnegative integer $n$ such that $F_{n+1}=\emptyset$ and $F_n\neq\emptyset.$
\end{enumerate}
\end{enumerate}
\label{theorem_comp_series}
\end{theorem}

Informally, this theorem states that a composition series exists exactly when the process of iteratively cutting the terminal vertices and the subsets of their breaking vertices ends after finitely many steps. If a graph has a composition series, the part of the proof showing (2) $\Rightarrow$ (1) provides an algorithm for obtaining a composition series of the graph. 

Before the proof, we consider the composition quotients in some examples.  
\begin{enumerate}
\item Let $E$ be the graph from part (1) of Example \ref{example_three_porcupine_quotients}. For this graph, $\Ter(E)$ is the saturated closure of the sinks and $\Ter(E)=E^0.$ Hence, $F_1=\emptyset.$ A composition series of $E$ can be obtained by considering the saturated closure of one of the sinks, then the saturated closure of that sink and another one, and, finally, the saturated closure $E^0$ of all three sinks. For example, by considering $\ol{\{w_0\}}$ first, we obtain the set $H$ from Example \ref{example_three_porcupine_quotients}. Considering the saturated closure of $H\cup \{v_0\}$ next, for example, produces the set $G$ from Example \ref{example_three_porcupine_quotients}. Lastly, the saturated closure of all three sinks is $E^0.$ This produces the composition series $\emptyset\leq H\leq G\leq E^0$ considered in Example \ref{example_three_porcupine_quotients}.  

\item  Let $E$ be the graph part (2) of Example \ref{example_three_porcupine_quotients}. For this graph, $\Ter(E)=\{w\}$ so that $F_1=E/(\{w\}, \{v\})$ is $\;\;\;\xymatrix{\bullet_v\ar@(lu,ld)}.$ As $\Ter(F_1)=\{v\}=F_1^0,$ $F_2=\emptyset.$ A composition series of $E$ can be produced by considering $\ol{\{w\}}=\{w\}$ without the breaking vertex $v$ of $\{w\}$ first, then $\{w\}$ together with the breaking set $\{v\}$, and, finally, adding the terminal vertex $v$ of $F_1$ to the set $\{w\}$ to obtain $\ol{\{v, w\}}=E^0.$ This produces the series $(\emptyset, \emptyset)\leq (\{w\}, \emptyset)\leq (\{w\}, \{v\})\leq E^0$ from Example \ref{example_three_porcupine_quotients}. 
  
\item If $E$ is the graph $\xymatrix{
\bullet_{u_1}\ar[r]&\bullet_{u_2}\ar[r]  & \bullet_{u_3}\ar@{.>}[r]& \\
\bullet_{v_1} \ar[r]\ar[u]        & \bullet_{v_2}\ar[r]\ar[u] & \bullet_{v_3}\ar[u]\ar@{.>}[r]&},$ then $\Ter(E)=\{u_n\mid n=1,2,\ldots\}$ so that $F_1$ is the graph $\xymatrix{
\bullet_{v_1} \ar[r] & \bullet_{v_2}\ar[r] & \bullet_{v_3}\ar@{.>}[r]&}$ and  $F_1^0=\Ter(F_1).$ So, $F_2=\emptyset.$ As all terminal vertices of $E$ are in the same cluster, $\Ter(E)$ can be taken to be the first term of a composition series. As  $F_1$ also has only one cluster, adding the terminal vertices of $F_1$ to $\Ter(E)$ produces the sequence $\emptyset\leq \Ter(E)\leq E^0$ which is a composition series of $E$.
\end{enumerate}

\begin{proof} 
(1) $\Rightarrow$ (2). If (1) holds, then conditions (a), (b), and (c) of Lemma \ref{lemma_for_comp_series} hold for $E=F_0$ by Lemma \ref{lemma_for_comp_series}. If $F_1=\emptyset,$ then (2) holds. If $F_1\neq\emptyset,$ then $F_1$ is a quotient of $F_0,$ so $F_1$ has a composition series by Proposition \ref{proposition_porcupine_and_quotient} and  (a), (b), and (c) of Lemma \ref{lemma_for_comp_series} hold for $F_1$ by Lemma \ref{lemma_for_comp_series}. Continuing these arguments, we obtain that  (a), (b), and (c) of Lemma \ref{lemma_for_comp_series} hold for $F_n$ for each $n$ such that $F_n\neq \emptyset.$ Hence, (2i) holds. 

For every $n$ such that $F_n\neq\emptyset,$  the vertices of $F_n$ are the vertices of $E$ only since the quotient used to form $F_n$ is taken with respect to the admissible pair with the set of all breaking vertices, so no new vertices are added when forming $F_n$ from $F_{n-1}$. Hence, $\Ter(F_n)\subseteq E^0.$ Let $H_0=\Ter(E)$ and $H_n=H_{n-1}\cup\Ter(F_n)$ for any $n$ such that $F_n$ is nonempty. Note that the saturated closure of the terminal vertices of $F_n$ is taken in $F_n,$ not in $E,$ so the set $H_n$ includes infinite emitters which are regular in $F_n$ and breaking vertices of $H_{n-1}.$ 
The set $H_n$ is hereditary in $E$ since every vertex of $H_n$ emits edges only to $H_i$ for $i\leq n.$ We claim that $H_n$ is also saturated in $E$. If $\ra(\so^{-1}(v))\subseteq H_n$ for a regular vertex $v\in E^0,$ then either $\ra(\so^{-1}(v))\subseteq H_{n-1}$ or $\ra(\so^{-1}(v))\cap \Ter(F_n)\neq \emptyset.$ In the first case, using inductive argument and the fact that $H_0$ is saturated, we conclude that $v\in H_{n-1}\subseteq H_n.$ In the second case, $v$ is a regular vertex of $F_n$ and the ranges of all edges $v$ emits in $F_n$ are in $\Ter(F_n).$ As $\Ter(F_n)$ is saturated in $F_n$, $v\in \Ter(F_n)\subseteq H_n.$ 

If $F_{n+1}\neq \emptyset,$ then $\Ter(F_{n+1})\neq\emptyset$ by Lemma \ref{lemma_for_comp_series}, so $H_n\subsetneq H_{n+1}.$ To show that $(H_n, B_{H_n}) \leq (H_{n+1}, B_{H_{n+1}}),$ it is sufficient to check that $B_{H_n}\subseteq H_{n+1}\cup B_{H_{n+1}}.$ If $v\in B_{H_n},$
then the set $\so^{-1}(v)\cap \ra^{-1}(E^0-H_n)$ is finite, nonempty and equal to the union of the mutually disjoint sets 
$\so^{-1}(v)\cap \ra^{-1}(\Ter(F_{n+1}))$  and $\so^{-1}(v)\cap \ra^{-1}(E^0-H_{n+1}).$ If the second set is nonempty, $v\in B_{H_{n+1}}.$ If the second set is empty, then $v$ is a regular vertex of $F_{n+1}$ which emits all its edges to $\Ter(F_{n+1}).$ As $\Ter(F_{n+1})$ is saturated in $F_{n+1}$, $v\in \Ter(F_{n+1})\subseteq H_{n+1}.$  

Since $E$ has a composition series, the chain $(\emptyset, \emptyset)\leq (H_0, B_{H_0})\leq \ldots$ eventually becomes constant by Corollary \ref{corollary_equivalence}.  If $n$ is the smallest such that $H_n=H_{n+1},$ then $\Ter(F_{n+1})=\emptyset$ which implies that  $F_{n+1}=\emptyset$ by part (2i). Since $H_{n-1}\subsetneq H_n,$ $\Ter(F_n)\neq\emptyset,$ so $F_n\neq\emptyset.$ This shows that (2ii) holds.
 
(2) $\Rightarrow$ (1). By (2ii), there is $n\geq 0$ such that $F_{n+1}=\emptyset$ and $F_n\neq\emptyset.$ Thus, $\Ter(F_n)=F_n^0\neq\emptyset.$ Since condition (b) of Lemma \ref{lemma_for_comp_series} holds for $F_n$, there are finitely many clusters. 
By Proposition  \ref{proposition_ter_ideal}, $L_K(F_n)$ is graded isomorphic to a finite sum of graded simple algebras. As such an algebra, $L_K(F_n)$ has a graded composition series. By Corollary \ref{corollary_composition_equivalent},  $F_n$ has a composition series. 

The condition (a) of Lemma \ref{lemma_for_comp_series} holds for $F_{n-1},$ so $\Ter(F_{n-1})\neq\emptyset.$ Since condition (b) of Lemma \ref{lemma_for_comp_series} holds, $\Ter(F_{n-1})$ has finitely many clusters. 
If $m$ is a positive integer, $C_i$ are the clusters of $F_{n-1}$ for $i=1,\ldots, m,$ $H_0=\emptyset$ and  $H_i=\ol{C_1\cup\ldots\cup C_i}$ for $i=1,\ldots, m,$ then $H_m=\Ter(F_{n-1})$ and the chain 
\[(\emptyset, \emptyset)=(H_0, \emptyset)\lneq (H_1, \emptyset)\lneq (H_2, \emptyset)\lneq \ldots \lneq (H_m, \emptyset)=(\Ter(F_{n-1}), \emptyset)\]
is a chain of admissible pairs of $F_{n-1}.$ As $I(H_{i+1})=I(H_i)\oplus I(C_{i+1})$ and $ I(C_{i+1})$ is graded simple by Proposition \ref{proposition_ter_ideal}, the porcupine-quotient $(H_{i+1}, \emptyset)/(H_i, \emptyset)$ is cofinal for each $i=0, \ldots, m-1.$ 

By condition (c) of Lemma \ref{lemma_for_comp_series} for $F_{n-1},$ $B_{\Ter(F_{n-1})}$ is finite. If $B_{\Ter(F_{n-1})}=\{v_1, \ldots, v_k\},$ let $S_0=\emptyset$ and $S_{i+1}=S_i\cup\{v_{i+1}\}$ for $i=0,\ldots, k-1.$ We have that $S_k=B_{\Ter(F_{n-1})}.$ Let us extend the above chain by \[(\Ter(F_{n-1}), \emptyset)=(\Ter(F_{n-1}), S_0) \lneq (\Ter(F_{n-1}), S_1)\lneq\ldots \lneq (\Ter(F_{n-1}), B_{\Ter(F_{n-1})}).\]
The porcupine-quotient graph 
$(\Ter(F_{n-1}), S_{i+1})/(\Ter(F_{n-1}), S_i)$ is an acyclic and row-finite graph with a unique sink $v_{i+1}$ and without infinite paths (Example \ref{example_with_B_H_and_emptyset} also establishes this), so part (3a) of Theorem \ref{theorem_graded_simple} holds by Lemma \ref{lemma_saturated_closure}. Hence, this porcupine-quotient is cofinal. 

Consider the graded ideals corresponding to the admissible pairs of the concatenation of the above two chains of admissible pairs. These ideals form a graded composition series of the algebra $I((\Ter(F_{n-1}), B_{\Ter(F_{n-1})})).$ By Corollary \ref{corollary_composition_equivalent}, the graph $P_{(\Ter(F_{n-1}), B_{\Ter(F_{n-1})})}$ has a composition series. Thus, we have that both the porcupine 
$P_{(\Ter(F_{n-1}), B_{\Ter(F_{n-1})})}$ and the quotient $F_n=F_{n-1}/(\Ter(F_{n-1}),$ $B_{\Ter(F_{n-1})})$ have composition series, so  $F_{n-1}$ has a composition series by Proposition \ref{proposition_porcupine_and_quotient}.  
Repeating these arguments shows that if $F_{i+1}$ has a composition series, then $F_i$ has a composition series for all $i$ starting with $i=n-2$ and ending with $i=0.$ Thus, $F_0=E$ has a composition series. 
\end{proof}

Theorem \ref{theorem_comp_series} has the following corollary. 

\begin{corollary}
Every unital Leavitt path algebra has a graded composition series.
\label{corollary_unital_comp_series}
\end{corollary}
\begin{proof}
Let $F_n$ for $n\geq 0$ be the composition quotients of $E$. Since $L_K(E)$ is unital, $E^0$ is finite and so $\Ter(E)$ is nonempty by Proposition \ref{proposition_existence_of_the_four_types} and the conditions (b) and (c) of Lemma \ref{lemma_for_comp_series} trivially hold. As $F_1$ is the quotient of $E$ with respect to an admissible pair with the entire breaking vertex set, $F_1$ also has finitely many vertices and so all three parts of Lemma \ref{lemma_for_comp_series} hold by the same argument. Continuing with such reasoning, we obtain that condition (2i) of Theorem \ref{theorem_comp_series} holds.

As $\Ter(E)=\Ter(F_0)$ is nonempty and no new vertices are added when forming $F_1,$ 
$|F_0^0|>|F_1^0|.$ Continuing applying the same argument, we have that $|F_i^0|>|F_{i+1}^0|$ for all $i$ such that $F_i\neq\emptyset.$ As $|E^0|$ is finite, there is a nonnegative integer $n$ such that $F_{n+1}=\emptyset.$ By taking smallest such $n,$ we have that $F_n\neq\emptyset.$ Thus, condition (2ii) of Theorem \ref{theorem_comp_series} holds. 
\end{proof}

The authors of \cite{Roozbeh_Alfilgen_Jocelyn} noted that if $E$ is finite, then $M_E^\Gamma$ has a composition series. 
By Corollaries \ref{corollary_unital_comp_series} and \ref{corollary_composition_equivalent_with_talented}, if $E$ has finitely many vertices (but possibly contains infinite emitters), then  $M_E^\Gamma$ has a composition series.

\section{Types of the talented monoids of cofinal porcupine-quotient graphs} 
\label{section_talented}

We recall that $\Gamma$ denotes the infinite cyclic group generated by an element $t.$
The monoid $M_E^\Gamma$ is cancellative (by \cite[Corollary 5.8]{Ara_et_al_Steinberg}) so the natural pre-order is, in fact, an order.
By \cite[Proposition 3.4]{Roozbeh_Lia_Comparable}, the relation $x< t^nx$ is impossible for any $x\in M_E^\Gamma$ and any positive integer $n.$ The remaining possibilities give rise to the following types.  
\begin{enumerate}
\item If $x=t^nx$ for some positive integer $n,$ we say that $x$ is {\em periodic}. 
 
\item If $x>t^nx$ for some positive integer $n,$ we say that $x$ is {\em aperiodic}. 

\item If $x$ and $t^nx$ are incomparable for any positive integer $n,$ we say that $x$ is {\em incomparable}.
\end{enumerate}
If $x$ is periodic or aperiodic, $x$ is {\em comparable}. This terminology matches the one used in  \cite{Roozbeh_Lia_Comparable}.
We note that \cite{Roozbeh_Alfilgen_Jocelyn} uses ``cyclic'' for ``periodic'' and ``non-comparable'' for ``incomparable''. In our terminology, the authors of \cite{Roozbeh_Alfilgen_Jocelyn} define a 
$\Gamma$-order-ideal $I$ of $M_E^\Gamma$ to be {\em periodic} (respectively, {\em comparable, incomparable}) if its every nonzero element is periodic (respectively, comparable, incomparable). We also say that $I$ is {\em aperiodic} if its every nonzero element is aperiodic. 

The proofs of Lemma \ref{lemma_order_ideals_generated_by_terminal} and Theorem \ref{theorem_periodic_aperiodic_incomp} use some results of \cite{Roozbeh_Lia_Comparable} and their corollaries which we summarize in the following proposition.

\begin{proposition}
Let $E$ be an arbitrary graph. 
\begin{enumerate}
\item \cite[Lemma 3.9 and Theorem 3.19]{Roozbeh_Lia_Comparable} If $x\in M_E^\Gamma$ is comparable, then there is a vertex $v$ on a cycle, a nonnegative integer $n,$ and $z\in M_E^\Gamma$ such that $x=t^ny+z$ where $y=[v]$ or $y=[q_Z^v].$ If $w$ is a vertex such that $[w]$ is comparable, then $w$ connects to a vertex in a cycle.

\item If $v\in E^0,$ then $[v]$ is a periodic element of $M_E^\Gamma$ if and only if $v$ is in the saturated closure of a finite set of vertices on cycles without exits. 

\item  If $v\in E^0$ is in the saturated  closure of a finite set of vertices on cycles, then $[v]$ is comparable. If at least one of those cycles has an exit, $[v]$ is aperiodic.

\item The element $[v]$ of $M_E^\Gamma$ is comparable for every $v\in E^0$ if and only if every $v\in E^0$ is in the saturated closure of a finite the set of vertices on cycles. 

\item \cite[Theorems 4.2 and 4.5 and Corollary 4.7]{Roozbeh_Lia_Comparable} The monoid $M_E^\Gamma$ is periodic (respectively, aperiodic or incomparable) if and only if $[v]$ is periodic (respectively, aperiodic or incomparable) for every vertex $v\in E^0.$
\end{enumerate}
\label{proposition_comparable_paper}
\end{proposition}
\begin{proof}
Parts (1) and (5) follow directly from the noted results of \cite{Roozbeh_Lia_Comparable}.   

By \cite[Theorem 4.1]{Roozbeh_Lia_Comparable}, $[v]$ is periodic for $v\in E^0$ if and only if any path originating at $v$ is a prefix of a path $p$ ending in one of finitely many cycles without exits and such that all vertices of $p$ are regular and every infinite path originating at $v$ ends in a cycle with no exits. This last condition is equivalent with $v$ being in the saturated closure of the vertices on finitely many cycles without exits by Lemma \ref{lemma_saturated_closure}. This shows that (2) holds.  

If the assumption of (3) holds, let $V$ be the set of vertices of finitely many cycles such that $v$ is in the saturated closure of $V.$ Then, there is a nonnegative integer $k$ such that $v\in \Lambda_k(V)$ where $\Lambda_k(V)$ are the sets from the paragraph before Lemma \ref{lemma_saturated_closure}. We can choose $k$ to be the smallest such that $v\in \Lambda_k(V).$ So, if $k>0,$ then $v\notin \Lambda_{k-1}(V)$. By the definition of $\Lambda_k(V),$ any element of $E^{\leq\infty}$ originating at $v$ contains an element of $V$ which shows that there are only finitely many paths originating at $v$ and terminating in a vertex of $V$ such that no vertex, except the range, is in $V.$ Let $n_v$ be the maximal element of the set of lengths of such paths and let $n_w$ be defined analogously for any $w\in \Lambda_i(V)$ for $i\leq k$. If $n_v=0,$ then $v\in V,$ so $v$ is on a cycle which implies that $[v]$ is comparable. If $v$ connects to a cycle with an exit, then one of the cycles in $V$ has to have an exit and $[v]$ is aperiodic by part (2). If $n_v>0,$ then $v$ is regular, $k>0,$ $\ra(\so^{-1}(v))\subseteq \Lambda_{k-1}(V),$ and the relation $n_{\ra(e)}<n_v$ holds for every $e\in \so^{-1}(v).$ Using induction, $[\ra(e)]$ is comparable, so $[\ra(e)]\geq t^{m_e}[\ra(e)]$ for some positive integer $m_e$  for every $e\in \so^{-1}(v).$ Let $m$ be the least common multiple of the elements of $\{m_e \mid e\in \so^{-1}(v)\}.$ Then $[\ra(e)]\geq t^m[\ra(e)]$ which implies that \[[v]=\sum_{e\in\so^{-1}(v)}t\,[\ra(e)]\geq \sum_{e\in\so^{-1}(v)}t\,t^m[\ra(e)]=t^m\sum_{e\in\so^{-1}(v)}t\,[\ra(e)]=t^m[v]\] so that $[v]$ is comparable.  
If at least one of the cycles with vertices in $V$ has an exit, then $[\ra(e)]$ is aperiodic for some $e\in \so^{-1}(v)$ and $[\ra(e)]>t^{m_e}[\ra(e)].$ Thus, $[v]>t^m[v],$ so $[v]$ is aperiodic. 

The implication ($\Rightarrow$) of (4) holds by \cite[Proposition 2.2, Lemma 3.9 and Theorem 3.21]{Roozbeh_Lia_Comparable} and ($\Leftarrow$) holds by part (3). 
\end{proof}

Lemma \ref{lemma_for_lemma} is used in the proof of Lemma \ref{lemma_order_ideals_generated_by_terminal} which is needed for Theorem \ref{theorem_periodic_aperiodic_incomp}.
Recall that the isomorphism of the lattice of admissible pairs of a graph $E$ and the lattice of $\Gamma$-order-ideals of $M^\Gamma_E$ maps $(H,S)$ to the $\Gamma$-order-ideal $J^\Gamma(H,S)$ generated by $\{[v]\mid v\in H\}\cup\{[v^H]\mid v\in S\}.$ The inverse isomorphism maps a $\Gamma$-order-ideal $I$ onto $(H,S)$ for  
$H=\{v\in E^0\mid [v]\in I\},$ and $S=\{v\in B_H\mid [v^H]\in I\}.$  

\begin{lemma}
If $E$ is any graph, $V$ is a set of vertices of $E,$ $H=\ol V,$ and $I$ is the $\Gamma$-order-ideal generated by $V,$ then $H=\{v\in E^0\mid [v]\in I\}$ and $\{v\in B_H\mid [v^H]\in I\}=\emptyset$ (i.e. $I=J^\Gamma(H, \emptyset)$). 
\label{lemma_for_lemma} 
\end{lemma}
\begin{proof}
Let $(G,S)$ be an admissible pair such that  $I=J^\Gamma(G,S).$ As $\{[v]\mid v\in V\}\subseteq I=J^\Gamma(G,S),$ $V\subseteq G.$ Since $H$ is the smallest hereditary and saturated set containing $V,$ $H\subseteq G.$ The converse holds since $V\subseteq H$ implies that $I\subseteq J^\Gamma(H,\emptyset).$ As $I=J^\Gamma(G,S),$ we have that $(G, S)\leq  (H, \emptyset)$ which implies $G\subseteq H$ and  $S\subseteq H.$ So, $G=H.$ As $S\subseteq E^0-G$ and $S\subseteq H=G,$ $S=\emptyset.$ 
\end{proof}

Lemma \ref{lemma_order_ideals_generated_by_terminal} describes the $\Gamma$-order-ideal generated by a cluster and shows that such an ideal is either periodic, aperiodic, or incomparable. Note that if $v$ is a vertex which is not terminal, then the $\Gamma$-order-ideal generated by $[v]$ can contain more than one type of elements. For example, if $E$ is the graph  
$\;\;\;\xymatrix{\bullet^u\ar@(lu,ld)  & \bullet^v\ar[l]\ar[r]&\bullet^w},$ then
the $\Gamma$-order-ideal generated by $[v]$ contains both $[u]$ and $[w],$ $[u]$ is periodic, and $[w]$ is incomparable.

Some parts of 
Lemma \ref{lemma_order_ideals_generated_by_terminal} generalize \cite[Theorems 3.10 and 3.11]{Roozbeh_Alfilgen_Jocelyn} shown for finite graphs.

\begin{lemma}
Let $E$ be any graph, $C$ be a cluster of a terminal vertex, and $I_C$ be the $\Gamma$-order-ideal of $M_E^\Gamma$ generated by $\{[v]\mid v\in C\}.$ The following holds. 
\begin{enumerate}[\upshape(1)]
\item The $\Gamma$-order-ideal $I_C$ is minimal and it is equal to the $\Gamma$-order-ideal generated by $[v]$ for any $v\in C.$  

\item If $v\in C$ is such that $[v]$ is periodic (respectively, aperiodic or incomparable), then $I_C$ is periodic (respectively, aperiodic or incomparable).

\item If $E$ is a cofinal, then $M_E^\Gamma$ is either periodic, aperiodic or incomparable: it is periodic if $C=c^0$ for a cycle $c$ without exits, aperiodic if $C=T(c^0)$ for an extreme cycle $c,$ and incomparable if $C$ does not contain a vertex on a cycle.   
\end{enumerate}  
\label{lemma_order_ideals_generated_by_terminal}
\end{lemma}
\begin{proof} 
By Lemma \ref{lemma_for_lemma}, $\ol C=\{v\in E^0\mid [v]\in I_C\}$ and $I_C=J^\Gamma(\ol C, \emptyset).$ By Lemma \ref{lemma_terminal_cluster}, for every $v\in C,$ $\ol{\{v\}}=\ol C$ which implies that $\ol C$ does not contain any nontrivial and proper hereditary and saturated subsets. Thus, $I_C$ is minimal and $J^\Gamma(\ol{\{v\}}, \emptyset)=I_C.$ Hence, part (1) holds. 

To show (2), let $v\in C.$ If $v$ is not on a cycle, $v$ is a sink or on a terminal path and no $u\in C$ connects to a cycle. Hence, no $u\in \ol C$ connects to a cycle 
and so $[u]$ is incomparable by part (1) of Proposition \ref{proposition_comparable_paper}. As $[w^p]=[\ra(p)]$ for $p\in F_1(\ol C,\emptyset),$ every vertex of $P_{(\ol C,\emptyset)}$ gives rise to an incomparable element of $M_{P_{(\ol C,\emptyset)}}^\Gamma.$ 
Thus, $I_C\cong M_{P_{(\ol C,\emptyset)}}^\Gamma$ is incomparable by part (5) of Proposition \ref{proposition_comparable_paper}.

If $v$ is on a cycle $c,$ then $c$ is either without exits or extreme. In the first case, $[u]$ is periodic for every $u\in \ol C$ by part (2) of Proposition \ref{proposition_comparable_paper}. This implies that $[w]$ is periodic for every vertex $w$ of $P_{(\ol C,\emptyset)}.$
Thus, $I_C\cong M_{P_{(\ol C,\emptyset)}}^\Gamma$ is periodic by part (5) of Proposition \ref{proposition_comparable_paper}. In the second case, every element of $C$ is on an extreme cycle and so $[u]$ is aperiodic for every $u\in \ol C$ by part (3) of Proposition \ref{proposition_comparable_paper}. Thus,
every element of $I_C\cong M_{P_{(\ol C,\emptyset)}}^\Gamma$ is aperiodic by part (5) of Proposition \ref{proposition_comparable_paper}. 

Part (3) holds by part (2) since the assumption that $E^0$ is cofinal is equivalent with $E^0=\ol C$ which implies that $M_E^\Gamma=I_C.$ The rest of the claim in (3) holds by the proof of part (2). 
\end{proof}

Theorem \ref{theorem_periodic_aperiodic_incomp} follows from Lemma \ref{lemma_order_ideals_generated_by_terminal}. If $E$ is a finite graph, parts (1a) and (3a) have been shown in 
\cite[Theorems 3.10 and 3.11]{Roozbeh_Alfilgen_Jocelyn}. 
We also note that (3c) have been stated in \cite[Corollary 4.7]{Roozbeh_Lia_Comparable}.  
 
\begin{theorem}
Let $E$ be any graph. The correspondence mapping a cluster $C$ of $E$ onto the $\Gamma$-order ideal $I_C$ generated by $\{[v]\mid v\in C\}$ (equivalently by $[v]$ for any $v\in C$) is 
a bijection mapping the set of clusters of $E$ onto the set of minimal $\Gamma$-order ideals. The following also holds. 
\begin{enumerate}[\upshape(1)]
\item 
\begin{enumerate}[\upshape(a)]
\item There is a bijection between the set of cycles of $E$ with no exits and the set of $\Gamma$-order-ideals of $M_E^\Gamma$ which are periodic and minimal.  

\item The $\Gamma$-order-ideal generated by the elements $[v]$ for $v$ a vertex in a cycle without exits is the largest periodic $\Gamma$-order-ideal of $M_E^\Gamma.$

\item The $\Gamma$-monoid $M_E^\Gamma$ is periodic if and only if $E^0$ is the saturated closure of the set of vertices on cycles with no exits. 
\end{enumerate}

\item 
\begin{enumerate}[\upshape(a)]
\item There is a bijection between the set of the clusters of vertices of $E$ on extreme cycles and the set of $\Gamma$-order-ideals of $M_E^\Gamma$ which are aperiodic and minimal. 

\item The $\Gamma$-monoid $M_E^\Gamma$ is aperiodic if and only
every cycle has an exit and every vertex of $E$ is in the saturated closure of a finite the set of vertices on cycles.
\end{enumerate}

\item 
\begin{enumerate}[\upshape(a)]
\item There is a bijection between the set of the clusters of vertices of $E$ which are either sinks or on terminal paths and the set of $\Gamma$-order-ideals of $M_E^\Gamma$ which are incomparable and minimal. 

\item The $\Gamma$-monoid $M_E^\Gamma$ is incomparable if and only if $E$ is acyclic. 
\end{enumerate}
\end{enumerate}
\label{theorem_periodic_aperiodic_incomp}
\end{theorem}
\begin{proof}
If $C$ is a cluster of $E$, the ideal $I_C$ is minimal and $I_C=J^\Gamma(\ol{\{v\}}, \emptyset)$ for any $v\in C$ by part (1) of Lemma \ref{lemma_order_ideals_generated_by_terminal}. The correspondence $C\mapsto I_C$ is injective since $I_C=I_D$ implies that $\{v\in E^0\mid [v]\in I_C\}=\{v\in E^0\mid [v]\in I_D\}.$ As $[v]\in I_C$ if and only if $v\in \ol C,$ and a similar equivalence holds for $D,$ we have that  
$\ol C=\ol D.$ Hence, $C\subseteq \ol D,$ so for any $v\in C,$ there is a path originating at $v$ and terminating at some $w\in D.$ Since $T(v)\subseteq T(C)=C,$ $w\in C\cap D$ which implies that $C=D.$ 

Next, we show that the correspondence $C\mapsto I_C$ is  onto. Let $I$ be a minimal ideal of $M_E^\Gamma.$ As $I$ is nontrivial, $[v]\in I$ for some $v\in E^0.$ If $[v]$ is periodic, then $v$ 
connects to a cycle $c$ without exits by part (2) of Proposition \ref{proposition_comparable_paper}. 
Thus, $\{[w]\mid w\in c^0\}\subseteq I$ and so the ideal $I_{c^0}$ generated by the set $\{[w]\mid w\in c^0\}$ is contained in $I.$ 
As $I$ is minimal, $I=I_{c^0}.$ 

If $[v]$ is aperiodic, $v$ connects to a cycle by part (1) of Proposition \ref{proposition_comparable_paper}. 
Assuming that all of the cycles to which $v$ connect have no exits, consider the hereditary and saturated closure $H$ of their vertices. As $\{0\}\subsetneq J^\Gamma(H, \emptyset)\subseteq I$ and $I$ is minimal, $J^\Gamma(H, \emptyset)=I$ which implies that $[v]\in J^\Gamma(H, \emptyset)$ so that $v\in H.$ By part (2) of Proposition \ref{proposition_comparable_paper},  $[v]$ is periodic. Since this is a contradiction, there is a cycle $c$ with an exit such that $c^0\subseteq T(v).$ So, $\{[w]\mid w\in c^0\}\subseteq I.$ Assuming that $c$ emits a path $p$ such that $\ra(p)\notin R(c^0),$ consider  the set $G=\ol{\{\ra(p)\}}.$ As $\so(p)\notin G,$ the $\Gamma$-order-ideal generated by $\{[w]\mid w\in G\}$ is nontrivial and strictly contained in $I.$ This is a contradiction, so no such path $p$ exists. Hence, $c$ is extreme. If $C$ is the cluster containing $c^0,$ $I_C\subseteq I$. As $I$ is minimal, $I=I_C.$ 

If $[v]$ is incomparable, $v$ is not on a cycle. If $v$ is a sink, then $I_{\{v\}}\subseteq I$ which implies that $I=I_{\{v\}}$ by the minimality of $I.$ If $v$ is not a sink, but $I$ contains $[w]$ for sink $w,$ $I=I_{\{w\}}$ by the same argument. Hence, we can consider the case when $I$ contains no element of the form $[w]$ for $w$ a sink. For any $w\in E^0$ such that $[w]\in I,$ $w$ connects only to vertices $u$ such that $[u]\in I,$ so $v$ does not connect to any cycles. If $v$ is an infinite emitter, then it is not on a cycle so the $\Gamma$-order-ideal generated by $[\ra(e)]$ for $e\in \so^{-1}(v)$ is a proper and nontrivial $\Gamma$-order-subideal of $I$. Since this cannot happen, $v$ is a regular vertex. As $v$ connects to neither sinks, infinite emitters, nor cycles, $v$ emits an infinite path $\alpha$ containing infinitely many vertices  
by Proposition \ref{proposition_existence_of_the_four_types}. If $\alpha$ is not terminal, a vertex of $\alpha$ emits an infinite path $\beta$ which emits a path $p$ such that $\ra(p)\notin R(\beta^0)$ which implies that no vertex of $\beta^0$ is in $T(\ra(p)).$ By Lemma \ref{lemma_saturated_closure}, 
the saturated closure $H$ of $T(\ra(p))$ does not contain $\so(p).$ Thus, the $\Gamma$-order-ideal generated by $\{[v]\mid v\in H\}$ is strictly contained in $I.$ As $I$ is minimal, this cannot happen, so $\alpha$ is terminal. If $C$ is the cluster containing $\alpha^0,$ this shows that $I_C\subseteq I.$ As $I$ is minimal, $I=I_C.$ This shows that the correspondence $C\mapsto I_C$ is onto. 

If $I$ is a minimal $\Gamma$-order-ideal and if $I=J^\Gamma(H,S)$ for some admissible pair $(H,S),$ then $P_{(H,S)}$ is cofinal, so $I\cong M_{P_{(H,S)}}^\Gamma$ is either periodic, aperiodic, or incomparable by part (3) of Lemma \ref{lemma_order_ideals_generated_by_terminal}.  This fact and the statement we just showed imply parts (1a), (2a), and (3a). 

To show (1b), let us recall that $\NE$  denotes the saturated closure of the set of vertices on cycles without exits. Let $I=J^\Gamma(\NE, \emptyset)$ so that $\NE=\{v\in E^0\mid [v]\in I\}.$ As $[v]$ is periodic for $v\in \NE$ by part (2) of Proposition \ref{proposition_comparable_paper},   $M_{P_{(\NE,\emptyset)}}^\Gamma\cong J^\Gamma(\NE, \emptyset)=I$ is periodic by part (5) of Proposition \ref{proposition_comparable_paper}. If $I'$ is a periodic $\Gamma$-order-ideal, then $[v]$ is periodic for every $v\in E^0$ such that $[v]\in I'.$ By part (2) of Proposition \ref{proposition_comparable_paper}, $v\in\NE.$ Thus, $[v]\in I,$ so $I'\subseteq I$. Hence, $I$ is the largest periodic $\Gamma$-order-ideal. 

Part (1c) follows from (1b) since $M_E^\Gamma$ is periodic if and only if 
$M_E^\Gamma$ is equal to $J^\Gamma(\NE, \emptyset)$ which is equivalent with $E^0=\NE.$

The direction ($\Rightarrow$) of part (2b) follows from parts (4) and (2) of Proposition \ref{proposition_comparable_paper} and the direction ($\Leftarrow$) from parts (3) and (5) of Proposition \ref{proposition_comparable_paper}. 

The direction $(\Rightarrow)$ of part (3b) is direct since $[v]$ is comparable if $v$ is on a cycle. The converse holds since  the existence of a nonzero comparable element implies the existence of a cycle by part (1) of Proposition 
\ref{proposition_comparable_paper}.  
\end{proof} 

In general, $M_E^\Gamma$ can contain elements of all three types. For example, let $E$ be the graph below. 
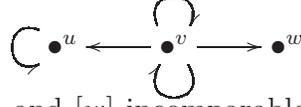
\[\xymatrix{\bullet^u\ar@(lu,ld)  & \bullet^v\ar[l]\ar[r]\ar@(ul,ur)\ar@(dr,dl)&\bullet^w}\] In $M_E^\Gamma,$ $[v]$ is aperiodic, $[u]$ periodic, and $[w]$ incomparable. Note that $E$ has a composition series $\emptyset \leq \{u\}\leq \{u,w\}\leq E^0$ and the talented monoids of the three corresponding porcupine-quotients are periodic, incomparable, and aperiodic respectively.
In Theorem \ref{theorem_two_types} and Corollary \ref{corollary_one_type}, we characterize graphs $E$ with the composition series of $M_E^\Gamma$ having composition factors of only two types and only one type. The authors of \cite{Roozbeh_Alfilgen_Jocelyn} studied conditions under which composition factors
of a composition series of $M_E^\Gamma$  for a finite graph $E$ are periodic or incomparable. \cite[Theorem 4.2]{Roozbeh_Alfilgen_Jocelyn}, without the part on Gelfand-Kirillov dimension, states that a finite graph has this property if and only if all its cycles are disjoint. Part (1) of Theorem \ref{theorem_two_types} implies this result for arbitrary graphs.

\begin{theorem}
Let $E$ be any graph.
\begin{enumerate}[\upshape(1)]
\item 
The following are equivalent.
\begin{enumerate}[\upshape(a)]
\item If $(H,S)$ and $(G,T)$ are admissible pairs of $E$ such that $(G,T)/(H,S)$ is cofinal, then  $M_{(G,T)/(H,S)}^\Gamma$ is either periodic or incomparable.

\item The cycles of $E$ are mutually disjoint. 
\end{enumerate}

\item The following are equivalent.  
\begin{enumerate}[\upshape(a)]
\item If $(H,S)$ and $(G,T)$ are admissible pairs of $E$ such that $(G,T)/(H,S)$ is cofinal, then $M_{(G,T)/(H,S)}^\Gamma$ is either aperiodic or incomparable.

\item Every cycle of $E$ contains a vertex of another cycle of $E.$   
\end{enumerate}

\item The following are equivalent.  
\begin{enumerate}[\upshape(a)]
\item If $(H,S)$ and $(G,T)$ are admissible pairs of $E$ such that $(G,T)/(H,S)$ is cofinal, then $M_{(G,T)/(H,S)}^\Gamma$ is either periodic or aperiodic.

\item Every vertex of $E$ is in the saturated closure of a finite set of  vertices on cycles.      

\item Every element of $M_E^\Gamma$ is periodic or aperiodic (i.e. comparable).
\end{enumerate}
\end{enumerate}
\label{theorem_two_types}
\end{theorem}
\begin{proof}
We show (1a) $\Rightarrow$ (1b) by contrapositive. Assume that $c$ is a cycle of $E$ which contains a vertex of another cycle of $E.$ If $(G,T)/(H,S)$ is a graph as in part (2) of Corollary \ref{corollary_construction_of_terminal_elements}, then it is cofinal and $c$ is extreme in it. By part (3) of Lemma \ref{lemma_order_ideals_generated_by_terminal}, $M^\Gamma_{(G,T)/(H,S)}$ is aperiodic. Thus, (1a) fails. 

Suppose that (1b) and the assumption of (1a) hold. As $(G,T)/(H,S)$ is cofinal, there is a unique cluster $C$ in $(G,T)/(H,S)$ by Theorem \ref{theorem_graded_simple}. Assume that $C$ contains vertices of an extreme cycle $c.$ Then $c^0\subseteq G-H$ because the vertices of $(G,T)/(H,S)$ which have the form $w^p$ or $v'$ or which are in 
$T-S$ are not on cycles. Since $c$ is extreme in $(G,T)/(H,S),$ there is an exit $e$ from $c$ such that $\ra(e)\in G-H.$ Since $\ra(e)$ connects back to $c$ in $(G,T)/(H,S),$ there is a cycle $d$ of $(G,T)/(H,S)$ which contains $e.$ Using the same argument as for $c^0\subseteq G-H,$ we have that $d^0\subseteq G-H$. Hence, $c$ and $d$ are cycles of $E$ which are not disjoint. This contradicts (1b), so either $C$ consists of vertices of a cycle without exits or $C$ contains no vertices on cycles. In the first case, $M_{(G,T)/(H,S)}^\Gamma$ is periodic and, in the second case, $M_{(G,T)/(H,S)}^\Gamma$ is incomparable by part (3) of Lemma \ref{lemma_order_ideals_generated_by_terminal}. 

We show (2a) $\Rightarrow$ (2b) by contrapositive. Assume that $c$ is a cycle of $E$ which contains a vertex of no other cycle of $E.$ If $(G,T)/(H,S)$ is a graph as in part (2) of Corollary \ref{corollary_construction_of_terminal_elements}, then it is cofinal and $c$ is without exists in it. By part (3) of Lemma \ref{lemma_order_ideals_generated_by_terminal}, $M^\Gamma_{(G,T)/(H,S)}$ is periodic. Thus, (2a) fails. 

To show (2b) $\Rightarrow$ (2a), assume that (2b) holds. Then any cofinal porcupine-quotient graph $(G,T)/(H,S)$ has no cycles without exits, so the set $\NE$ of $(G,T)/(H,S)$ is empty. By part (1b) of Theorem \ref{theorem_periodic_aperiodic_incomp}, no element of $M_{(G,T)/(H,S)}^\Gamma$ is periodic. By part (3) of Lemma \ref{lemma_order_ideals_generated_by_terminal}, $M_{(G,T)/(H,S)}^\Gamma$ is either aperiodic or incomparable. This shows (2a).

We show (3a) $\Rightarrow$ (3b) by contrapositive. If there is a vertex which is not in the saturated closure of finitely many vertices on cycles, then it emits a path to either a sink $v$, an infinite emitter $v$ which is not on a cycle, or it is on an infinite path $\alpha$ such that $T(\alpha^0)$ contains neither sinks, infinite emitters nor vertices on cycles by Lemma \ref{lemma_saturated_closure}. In the first two cases, let $(G,T)/(H,S)$ be a graph as in part (1) of Corollary \ref{corollary_construction_of_terminal_elements}. So, $(G,T)/(H,S)$ is cofinal and $v$ is its sink. By part (3) of Lemma \ref{lemma_order_ideals_generated_by_terminal}, $M_{(G,T)/(H,S)}^\Gamma$ is incomparable. Thus, (3a) fails. 
In the third case, let $(G,T)/(H,S)$ be a graph as in part (3) of Corollary \ref{corollary_construction_of_terminal_elements}. So, $(G,T)/(H,S)$ is cofinal and $\alpha$ is its terminal path. By part (3) of Lemma \ref{lemma_order_ideals_generated_by_terminal}, $M_{(G,T)/(H,S)}^\Gamma$ is incomparable. Thus, (3a) fails.

To show (3b) $\Rightarrow$ (3a), assume that (3b) holds and that $(G,T)/(H,S)$ is cofinal. By Theorem \ref{theorem_graded_simple}, there is a unique cluster $C$ of $(G,T)/(H,S)$ such that $((G,T)/(H,S))^0=\ol{C}.$ By (3b), there are neither sinks nor terminal paths in $(G,T)/(H,S),$ so $C$ contains a cycle $c.$ If $c$ is without exits in $(G,T)/(H,S),$ then $M_{(G,T)/(H,S)}^\Gamma$ is periodic and if $c$ is extreme, $M_{(G,T)/(H,S)}^\Gamma$ is aperiodic by part (3) of Lemma \ref{lemma_order_ideals_generated_by_terminal}.
 
The equivalence of (3b) and (3c) holds by parts (4) and (5) of Proposition \ref{proposition_comparable_paper}. 
\end{proof}

The condition that no cycle of a graph $E$ has an exit is strictly stronger than part  (1b) of Theorem \ref{theorem_two_types}. 
By \cite[Corollary 4.8]{Roozbeh_Lia_Comparable}, every element of $M_E^\Gamma$ is periodic or incomparable if and only if no cycle of $E$ has an exit. We also have that condition (2b) of Theorem \ref{theorem_two_types} is strictly stronger than the condition that each cycle of a graph $E$ has an exit. This last condition is equivalent to every nonzero element of $M_E^\Gamma$ being aperiodic or incomparable by \cite[Corollary 4.3]{Roozbeh_Lia_Comparable}. The equivalence of parts (3a) and (3c) of Theorem \ref{theorem_two_types} contrasts the strictness of the two implications mentioned above.  

Theorems \ref{theorem_periodic_aperiodic_incomp} and \ref{theorem_two_types} have the following corollary. 

\begin{corollary}
Let $E$ be any graph.  
\begin{enumerate}[\upshape(1)]
\item The following are equivalent.
\begin{enumerate}[\upshape(a)]
\item If $(H,S)$ and $(G,T)$ are admissible pairs of $E$ such that $(G,T)/(H,S)$ is cofinal, then $M_{(G,T)/(H,S)}^\Gamma$ is periodic.

\item The cycles of $E$ are mutually disjoint and every vertex of $E$ is in the saturated closure of a finite set of  vertices on cycles.
\end{enumerate}

\item The following are equivalent.
\begin{enumerate}[\upshape(a)]
\item If $(H,S)$ and $(G,T)$ are admissible pairs of $E$ such that $(G,T)/(H,S)$ is cofinal, then $M_{(G,T)/(H,S)}^\Gamma$ is aperiodic.

\item  Every cycle of $E$ contains a vertex of another cycle of $E$ and every vertex of $E$ is in the saturated closure of a finite set of vertices on cycles.
\end{enumerate} 

\item The following are equivalent.
\begin{enumerate}[\upshape(a)]
\item If $(H,S)$ and $(G,T)$ are admissible pairs of $E$ such that $(G,T)/(H,S)$ is cofinal, then $M_{(G,T)/(H,S)}^\Gamma$ is incomparable.

\item The graph $E$ is acyclic.
\end{enumerate} 
\end{enumerate}
\label{corollary_one_type}
\end{corollary}
\begin{proof}
Parts (1) and (2) follow directly from Theorem \ref{theorem_two_types}. If $E$ has a cycle $c,$ then there are admissible pairs $(G,T)$ and $(H,S)$ such that 
$(G,T)/(H,S)$ is cofinal and $(G,T)/(H,S)$  contains $c$ by part (2) of Corollary \ref{corollary_construction_of_terminal_elements}. 
Hence, $M_{(G,T)/(H,S)}^\Gamma$ is comparable by part (3) of Lemma \ref{lemma_order_ideals_generated_by_terminal}. If $E$ is acyclic, then $(G,T)/(H,S)$ is acyclic for any $(G,T)$ and $(H,S)$ such that $(G,T)/(H,S)$ is cofinal. Thus, $M_{(G,T)/(H,S)}^\Gamma$ is incomparable by part (3b) of Theorem \ref{theorem_periodic_aperiodic_incomp} (also by part (3) of Lemma \ref{lemma_order_ideals_generated_by_terminal}).    
\end{proof}

\end{document}